\documentclass{article}
\usepackage[all,ps,arc]{xy}  
\usepackage{amsmath,amssymb,latexsym}
\usepackage{amsfonts}
\usepackage{amsthm}
\usepackage{xcolor}

\newcommand{\point}{\ensuremath{\xymatrix{A\ar@<+.6ex>[r]^(.5){\alpha}
&B\ar@<+.6ex>[l]^(.5){\beta}}}}
\newcommand{\rg}{\ensuremath{\xymatrix{A\ar@<+1ex>[r]^{\alpha}\ar@<-1ex>[r]_{\gamma}&B\ar[l]|{\beta}}}}

\newcommand{\SH}{{\rm (SH)}}

\newtheorem{Theorem}{Theorem}[section]
\newtheorem{Lemma}[Theorem]{Lemma}
\newtheorem{Proposition}[Theorem]{Proposition}
\newtheorem{Definition}[Theorem]{Definition}

\theoremstyle{remark}\newtheorem{Remark}[Theorem]{Remark}
\newtheorem{Example}[Theorem]{Example}

\newtheorem{Problem}[Theorem]{Problem}

\newcommand{\bG}{\ensuremath{\mathbb G}}

\newcommand{\bH}{\ensuremath{\mathbb H}}
\newcommand{\bK}{\ensuremath{\mathbb K}}

\newcommand{\cA}{\ensuremath{\mathsf A}}
\newcommand{\cB}{\ensuremath{\mathsf B}}
\newcommand{\cC}{\ensuremath{\mathsf C}}

\newcommand{\cE}{\ensuremath{\mathsf E}}

\newcommand{\cQ}{\ensuremath{\mathsf Q}}
\newcommand{\cS}{\ensuremath{\mathsf S}}

\newcommand{\cM}{\ensuremath{\mathsf M}}

\newcommand{\cX}{\ensuremath{\mathsf X}}

\newcommand{\cY}{\ensuremath{\mathsf Y}}

\newcommand\Aut{\ensuremath{\text{Aut}}}

\newcommand\Out{\ensuremath{\text{Out}}}

\newcommand\coker{\ensuremath{\mathrm{coker\,}}}

\newcommand\Gpd{\ensuremath{\mathsf{Gpd}}}
\newcommand\Frac{\ensuremath{\mathsf{Frac}}}
\newcommand\Set{\ensuremath{\mathsf{Set}}}
\newcommand\Gp{\ensuremath{\mathsf{Gp}}}

\newcommand\XExt{\ensuremath{\mathsf{XExt}}}
\newcommand\BExt{\ensuremath{\mathsf{BExt}}}

\newcommand\OPEXT{\ensuremath{\mathsf{OPEXT}}}

\newcommand\Bfly{\ensuremath{\mathsf{Bfly}}}
\newcommand\XMod{\ensuremath{\mathsf{XMod}}}
\newcommand\Mod{\ensuremath{\mathsf{Mod}}}
\newcommand\Mon{\ensuremath{\mathsf{Mon}}}
\newcommand\Cat{\ensuremath{\mathsf{Cat}}}
\newcommand\Ab{\ensuremath{\mathsf{Ab}}}
\newcommand\AutM{\ensuremath{\mathsf{AutM}}}
\newcommand\AM{\ensuremath{\mathsf{AM}}}
\newcommand\AssAlg{\ensuremath{\mathsf{AssAlg}}}
\newcommand\Bimod{\ensuremath{\mathsf{Bimod}}}
\newcommand\XBiext{\ensuremath{\mathsf{XBiext}}}
\newcommand\BXBiext{\ensuremath{\mathsf{BXBiext}}}
\newcommand\Vect{\ensuremath{\mathsf{Vect}}}
\newcommand{\Fib}{\ensuremath{\mathsf{Fib}}}
\newcommand{\OpFib}{\ensuremath{\mathsf{OpFib}}}

\newcommand{\uffa}{\ensuremath{\underline\varphi}}

\newdir{|>}{{}*!/8pt/@{|}*!/3.5pt/:(1,-.2)@^{>}*!/3.5pt/:(1,+.2)@_{>}}
\newdir{ >}{{}*!/-10pt/@{>}}
\newdir{ |>}{!/10pt/{}*!/4.5pt/@{|}*:(1,-.2)@^{>}*:(1,+.2)@_{>}}

\begin{document}

\def \tm{\!\times\!}

\newenvironment{changemargin}[2]{\begin{list}{}{
\setlength{\topsep}{0pt}
\setlength{\leftmargin}{0pt}
\setlength{\rightmargin}{0pt}
\setlength{\listparindent}{\parindent}
\setlength{\itemindent}{\parindent}
\setlength{\parsep}{0pt plus 1pt}
\addtolength{\leftmargin}{#1}\addtolength{\rightmargin}{#2}
}\item}{\end{list}}

\title{Fibred categorical theory of\\ obstruction and classification of morphisms}
\author{A.\ S.\ Cigoli, S.\ Mantovani, G.\ Metere and E.\ M.\ Vitale}
\maketitle

\begin{abstract}
We set up a fibred categorical theory of obstruction and classification of morphisms that specializes to the one of monoidal functors between categorical groups and also to the Schreier-Mac Lane theory of group extensions. Further applications are provided, as for example a classification of unital associative algebra extensions with non-abelian kernel in terms of Hochschild cohomology. 

\end{abstract}

\section{Introduction}

Any extension
\[
\xymatrix{
    0 \ar[r] & K \ar[r]^k & E \ar[r]^f & C \ar[r] & 0
}
\]
of groups determines an action of $E$ on $K$, and in turn a
homomorphism $\psi_0 \colon C \to \frac{\mathrm{Aut}(K)}{\mathrm{Inn}(K)}=\mathrm{Out}(K)$, called the \emph{abstract kernel} of the extension. It is a classical problem to establish whether, given a morphism $\psi_0$ as above, there exists an extension having $\psi_0$ as its abstract kernel. The answer to this question is provided by the Schreier--Mac Lane Theorem (see \cite{Homology}), where it is proved that each abstract kernel determines a corresponding action $\xi$ of $C$ on the centre $\mathrm{Z}(K)$ of $K$, and an element of $H^3(C,\mathrm{Z}(K),\xi)$ called \emph{obstruction}. The requested extension exists if and only if the obstruction vanishes. Moreover, if the obstruction vanishes, then the set of extensions inducing the given abstract kernel is a simply transitive $H^2(C,\mathrm{Z}(K),\xi)$-set.

Two remarkable generalizations of the Schreier--Mac Lane obstruction theory are known. The first one, due to Bourn (see \cite{Bourn08}), basically shows that Schreier--Mac Lane Theorem still holds in a wider class of categories, such as a semi-abelian category with suitable properties. The second one is based on the homotopy classification of categorical groups established by Sinh in \cite{Sinh} and is stated in a more explicit way by Cegarra, Garc\'ia-Calcines and Ortega in \cite{CGCO}. It consists in replacing extensions of groups by monoidal functors between categorical groups. Indeed, extensions $(f,k)$ as above bijectively correspond to monoidal functors from the categorical group associated with the crossed module $0 \to C$ to the one associated with $K \to \mathrm{Aut}(K)$.

Since these two generalizations go into quite different directions, in this paper we adress the problem of finding a general setting which subsumes at the same time the semi-abelian setting and the categorical group setting.

In order to understand the solution we propose, let us look at the point of view on the \emph{obstruction problem} for (crossed) extensions of groups adopted in \cite{CM16}. The category $\XExt(\Gp)$ of crossed extensions of groups is equipped with a functor $\Pi\colon\XExt(\Gp) \to \Mod$, which sends each crossed extension
\[
\xymatrix{
0 \ar[r] & B \ar[r] & G_2 \ar[r]^{\partial} & G_1 \ar[r] & C \ar[r] & 0
}
\]
to the group module $(C,B)$, the action of $C$ on $B$ being induced by the crossed module $\partial$. Weak equivalences in $\XExt(\Gp)$ are those morphisms of crossed extensions which are turned into isomorphisms by the functor $\Pi$. Therefore, $\Pi$ factorizes through the corresponding category of fractions, whose morphisms are isomorphism classes of the so-called \emph{butterflies} (see Sections \ref{sec:XExt}--\ref{sec:bfly_gp})
\[
\xymatrix{
\XExt(\Gp)\ar[r]^-Q\ar[dr]_{\Pi}
&[\BExt](\Gp)\ar[d]^{P}
\\
&\Mod}
\]
An extension
\[
\xymatrix{
    0 \ar[r] & K \ar[r]^k & E \ar[r]^f & C \ar[r] & 0
}
\]
inducing the abstract kernel $\psi_0 \colon C \to \mathrm{Out}(K)$ gives rise to the butterfly depicted by the following commutative diagram, where the left and right columns are crossed extensions.
$$
\xymatrix@!C=6ex{
	0 \ar[d] \ar[rr] & &  \text{Z}(K) \ar[d] \\
	0 \ar[rd] \ar[dd] & & K \ar[dd]^{\mathcal I_K} \ar[ld]_{k} \\
	& E \ar[ld]_{f} \ar[rd]|-{\textsf{conj}} \\
	C \ar[d]_1 & & \Aut(K) \ar[d] \\
	C \ar[rr]_-{\psi_0} & & \Out(K)
}
$$
This interpretation allows us to translate the existence of the extension $(f,k)$ into the existence of such a morphism in the category $[\BExt](\Gp)$.

Having in mind this particular case, one can generalize the classical obstruction problem as follows: given any two crossed extensions $X$ and $X'$ and a morphism $(\psi_0,\psi)$ between their associated modules via $\Pi$, is there a butterfly between $X$ and $X'$ whose image via $P$ is $(\psi_0,\psi)$?

Actually, the above factorization of $\Pi$ through the category of fractions $[\BExt](\Gp)$ lives in the 2-category $\underline\Fib(\Gp)$ of fibrations over \Gp. It is proved in \cite{CMMV_Yoneda} that $\Pi$ is a fibrewise opfibration in $\underline\Fib(\Gp)$ (see Definition \ref{def:FOF}), and this implies that in the above factorization, $P$ is a fibrewise opfibration as well, but with groupoidal fibres (see Proposition 4.8 in \cite{refl_fib}). 
Moreover, thanks to the existence of liftings and coliftings in the fibrewise opfibration $P$, the obstruction problem formulated above may be reduced to the case where $(\psi_0,\psi)$ is an identity. 

This points the way to a formal context where to develop an abstract obstruction theory. In fact, we can formulate an \emph{obstruction problem} as in \ref{def:obstruction_problem}:
\begin{itemize}
 \item Let 
 $$ \xymatrix{
    \cX \ar[rr]^{P}\ar[dr]_{F}
    &&\cM\ar[dl]^{G}
    \\
    &\cB
    }$$
 be a fibrewise opfibration in $\Fib(\cB)$. Given two objects $x$ and $y$ of $\cX$, and a morphism $\varphi\colon P(x)\to P(y)$ of $\cM$, is there any $f\colon x\to y$ in $\cX$ such that $P(f)=\varphi$? When this is the case, is it possible to describe the set of such morphisms?
\end{itemize}

The main result of this work is Theorem \ref{thm:structure}, where we show that also in such a formal context, if $P$ has groupoidal fibres, the set of solutions of a given obstruction problem, if not empty, is still a simply transitive $\Gamma$-set, where $\Gamma$ is the automorphism group of an object in the fibre, and the action is given by arrow composition.

Once we get such a formulation in an abstract context, we can apply directly the above result to concrete situations, that can be described by means of a fibrewise opfibration with groupoidal fibres, such as:
\begin{itemize}
 \item extensions of groups with abelian kernel (Theorem \ref{thm:classification_gp});
 \item abelian extensions in a semi-abelian context (Theorem \ref{thm:opext_semiab});
 \item singular extensions of unital associative algebras (Theorem \ref{thm:classification_aa}).
\end{itemize}

Moreover, there are interesting situations described by a fibrewise opfibration whose fibres are not necessarily groupoids, so that Theorem \ref{thm:structure} cannot be applied directly. However, as explained in \cite{refl_fib}, under suitable conditions one can factorize a fibrewise opfibration through an appropriate category of fractions and the resulting factorization is a fibrewise opfibration with groupoidal fibres (see Proposition \ref{prop:factorization}). This way, we translate the original obstruction problem into a different problem, where we look for \emph{weak maps} (i.e.\ morphisms in the category of fractions), instead of maps, between the same objects. In this framework, we can apply our general result to different cases, namely:
\begin{itemize}
\item the already mentioned butterflies between crossed extensions of groups (Theorem \ref{thm:XExt});
\item monoidal functors between categorical groups: we recover a cohomological classification of such functors in terms of homotopy invariants of categorical groups (Theorem \ref{thm:monoidal_functors_CG}); 
\item crossed bimodule butterflies, introduced by Aldrovandi in \cite{Aldrovandi}: we get the classification Theorem \ref{thm:structure_AssAlg}, where we take advantage of the description of the third Hochschild cohomology groups in terms of crossed biextensions provided in \cite{BauMi02};
\item as a particular case of crossed bimodule butterflies, we obtain a variation of the Schreier-Mac Lane Theorem, providing a classification of unital associative algebra extensions with non-abelian kernel (Theorem \ref{thm:SML_AA}).
\end{itemize}

\section{Preliminaries}

For the sake of completeness, in this section we recall the following well-known facts and definitions about fibrations of categories.

\begin{Definition}\label{def:cartesian}
Let $P\colon \cX\to \cB$ be a functor.
A morphism $f\colon x\to y$ is called  \emph{cartesian}  w.r.t.\ $P$, or $P$-cartesian, if
\begin{itemize}
\item for all $\alpha\colon a'\to a$ and $f'\colon x'\to y$ with $P(f')=P(f)\cdot \alpha$,
there is a unique lifting $\hat\alpha\colon x'\to x$ with $P(\hat\alpha)=\alpha$ and
    $f'=f\cdot \hat\alpha$.
\end{itemize}

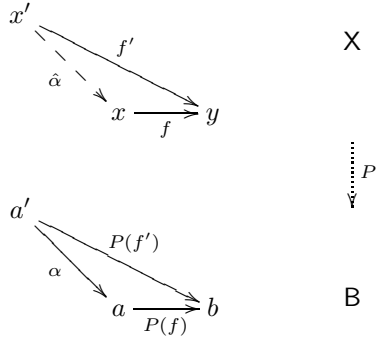
$$
\begin{aligned}
\xymatrix{
x'\ar@{-->}[dr]_{\hat\alpha}\ar[drr]^{f'}
\\
&x\ar[r]_{f}
&y
\\
a'\ar[dr]_{\alpha}\ar[drr]^{P(f')}
\\
&a\ar[r]_{P(f)}
&b
}
\end{aligned}
\qquad
\qquad
\begin{aligned}
\xymatrix{\cX\\\ar@{.>}[d]^{P}\\\\\cB}
\end{aligned}
$$
A morphism $f\colon x\to y$ is called $P$-vertical (or just vertical) if $P(f)$ is an identity.

\smallskip
For $b\in \cB$, we denote by $\cX_b$ the fibre of $P$ over $b$, i.e.\ the subcategory of $\cX$ determined by all morphisms $g$ such that
$P(g)= 1_b$.
\end{Definition}

We can extend the fibre notation: for $\alpha\colon P(x')\to P(x)$ we
write $\cX_{\alpha}(x',x)$ for the set of those morphisms
$h\colon x'\to x$ with $P(h)=\alpha$.

\begin{Lemma}\label{lem:cartesian_bijection}
With notations as above, the morphism $f\colon x\to y$ is cartesian
if and only if for all $x'$ in $\cX$ and $\alpha\colon P(x') \to P(x)$, the map given by the composition with $f$ yields a bijection:
$$
f\cdot - \ \colon \cX_{\alpha}(x',x)\to \cX_{P(f)\cdot\alpha}(x',y)\,.
$$
\end{Lemma}

\begin{Definition} \label{def:fibration}
The functor $P\colon \cX\to \cB$ is a fibration, or fibration over $\cB$, if for every $\varphi\colon a \to b$ in \cB, and
$y$ object of $\cX_b$ there is a cartesian lifting $f\colon x\to y$ of $\varphi$ at $y$, i.e.\ $f$ is $P$-cartesian and $P(f)=\varphi$.
\end{Definition}

Cartesian liftings  are universal: if a morphism $\varphi$ admits two cartesian liftings $f$ and $f'$ at the same $y$, then the unique comparison $h$ with $f'=f\cdot h$ is an isomorphism.

\begin{Definition} \label{def:Fib(B)}
Let $P\colon \cX\to\cB$ and $Q\colon \cY\to \cB$ be fibrations over $\cB$. A
functor $F\colon \cX\to \cY$ is said to be \emph{fibred (or cartesian) over $\cB$}  if
\begin{enumerate}
\item $Q\cdot F=P$,
\item if $f$ is $P$-cartesian, then $F(f)$ is $Q$-cartesian.
\end{enumerate}

\noindent Given two fibred functors  $F$ and $G$  from $P$ to $Q$, a natural transformation
$$\tau \colon F\Rightarrow G$$ is said to be \emph{vertical over $\cB$} if, for every object $x$ in $\cX$, its components $\tau_x\colon F(x)\to G(x)$ are $Q$-vertical.

\smallskip
\noindent These data define the 2-category $\Fib(\cB)$.

\smallskip
We say that $P\colon \cX\to \cB$ is an opfibration if $P^\mathsf{op}\colon \cX^\mathsf{op}\to \cB^\mathsf{op}$
is a fibration. 
The related notions of opcartesian morphisms and opcartesian liftings are understood. Opfibrations, opfibred functors and vertical natural transformations define the 2-category $\OpFib(\cB)$.
\end{Definition}

\section{The obstruction problem and the classification theorem}\label{sec:setting}
%
%
In the present section, we start by recalling the definition of \emph{fibrewise opfibration}, the formal setting where our obstruction problems take place. Then we state our main result as a classification theorem that, under suitable conditions, describes the solutions of a given obstruction problem.


\subsection{Fibrewise opfibrations and the obstruction problem} \label{sec:problem}

\begin{Definition}[\cite{CMMV_Yoneda}] \label{def:FOF}
A morphism  $P\colon(\cX,F)\to(\cM,G)$ in $\Fib(\cB)$
\begin{equation} \label{diag:FOF}
  \begin{aligned}
    \xymatrix{
    \cX \ar[rr]^{P}\ar[dr]_{F}
    &&\cM\ar[dl]^{G}
    \\
    &\cB
    }
  \end{aligned}
\end{equation}
is a \emph{fibrewise opfibration} if for every object $b$ in $\cB$, the restriction to the fibres
$$
\xymatrix{P_{b}\colon \cX_b\ar [r] &\cM_b}
$$
is an opfibration.
\end{Definition}


\begin{Example} \hfill
\begin{enumerate} 
\item Any internal opfibration in $\Fib(\cB$) is an example of fibrewise opfibration (see \cite{CMMV_Yoneda} for details). 
\item Let $S\colon\cS\to \cA\times\cB$ be a Yoneda regular span. Then the diagram
$$
\xymatrix{
    \cS \ar[rr]^{S}\ar[dr]_{S_2=P_2\cdot S}
    &&\cA\times \cB\ar[dl]^{P_2}
    \\
    &\cB
    }
$$
is a fibrewise opfibration. 
\item A special case is when $S$ is a $2$-sided fibration, or a discrete $2$-sided fibration, as for instance when it is determined by a profunctor.
\end{enumerate}
\end{Example}

The main object of our study may be formalized in the following way.

\begin{Problem} \label{def:obstruction_problem}
Let us suppose we are given a fibrewise opfibration $P\colon(\cX,F)\to(\cM,G)$,  two objects $x$ and $y$ of $\cX$, and a morphism $\varphi\colon P(x)\to P(y)$ of $\cM$.
The \emph{obstruction problem} associated with the triple $(x,y,\varphi)$ is to investigate whether there exists any $f\colon x\to y$ such that $P(f)=\varphi$, and in this case, to describe the set of such morphisms.
\end{Problem}

\subsection{The classification theorem}
Given a fibrewise opfibration (Definition \ref{def:FOF})
\begin{equation}
  \begin{aligned}
    \xymatrix{
    \cX \ar[rr]^{P}\ar[dr]_{F}
    &&\cM\ar[dl]^{G}
    \\
    &\cB
    }
  \end{aligned}
\end{equation}
let us consider two objects $x$ and $y$ of $\cX$,
and a morphism $\varphi\colon P(x)\to P(y)$.

Since $F$ is a fibration, there exists a cartesian lifting
$$
\xymatrix{w\colon\varphi^*y\ar[r]&y}
$$
of $G(\varphi)$ at $y$.
Then, since $P$ is fibred over $\cB$, $\varphi_k=P(w)$ is cartesian w.r.t.\ $G$,
thus giving a factorization
$$
\varphi = \varphi_k\cdot\varphi_v
$$
with $\varphi_v$ lying in the fibre $\cM_{F(x)}$.

Let $P_{F(x)}\colon \cX_{F(x)}\to \cM_{F(x)}$ be the restriction of $P$ to the
fibres over $F(x)$. By hypothesis, $P_{F(x)}$ is an opfibration, so that we can exhibit an opcartesian lifting $u$ of $\varphi_v$ at $x$:
$$
\xymatrix{
x\ar[r]^-u
&\varphi_*x
\\
P(x)\ar[r]_-{\varphi_v}
&P(\varphi^* y)
}
\qquad\qquad
\xymatrix{
\cX_{F(x)}\ar[d]^{P_{F(x)}}
\\
\cM_{F(x)}
}
$$
By Lemma \ref{lem:cartesian_bijection}, and by its dual version, we obtain:
\begin{enumerate}
\item a bijection given by the composition with $w$:
$$
\xymatrix{w\cdot - \colon \cX_{F(x)}(x,\varphi^*y)\ar[r]
&\cX_{G(\varphi)}(x,y)}
$$
\item a bijection given by precomposition with $u$:
$$
\xymatrix{-\cdot u \colon \left(\cX_{F(x)}\right)_{P(\varphi^*y)}(\varphi_*x,\varphi^*y)\ar[r]
&\left(\cX_{F(x)}\right)_{\varphi_v}(x,\varphi^*y)}
$$
The reader will be easily convinced that no ambiguity will arise if we write the last bijection as follows:
$$
\xymatrix{-\cdot u \colon \cX_{P(\varphi^*y)}(\varphi_*x,\varphi^*y)\ar[r]
&\cX_{\varphi_v}(x,\varphi^*y)}
$$
\end{enumerate}
Indeed
$$
\cX_{\varphi_v}(x,\varphi^*y)=
\left(\cX_{F(x)}\right)_{\varphi_v}(x,\varphi^*y)\subseteq
\cX_{F(x)}(x,\varphi^*y)\,,
$$
so that we can restrict the first bijection to
\begin{itemize}
\item[3.] a bijection
$$
\xymatrix{w\cdot - \colon \cX_{\varphi_v}(x,\varphi^*y)\ar[r]
&\cX_{P(w)\cdot\varphi_v}(x,y)}
$$
\end{itemize}

The previous discussion leads to the following statement.
\begin{Theorem} \label{thm:obstruction}
In the fibrewise opfibration $(P,F,G)$, we consider two objects $x$ and $y$ in $\cX$, and a map $\varphi\colon P(x)\to P(y)$. Then there is a bijection
$$
\xymatrix{
\Phi=\Phi_{x,y,\varphi}\colon\cX_{P(\varphi^*y)}(\varphi_*x,\varphi^*y)\ar[r]
&\cX_{\varphi}(x,y)
}
$$
\end{Theorem}

\begin{proof}
Since $P(w)\cdot\varphi_v=\varphi_k \cdot\varphi_v= \varphi$, we can compose the maps 2.\ and 3.\ above,
and get the bijection $\Phi$.
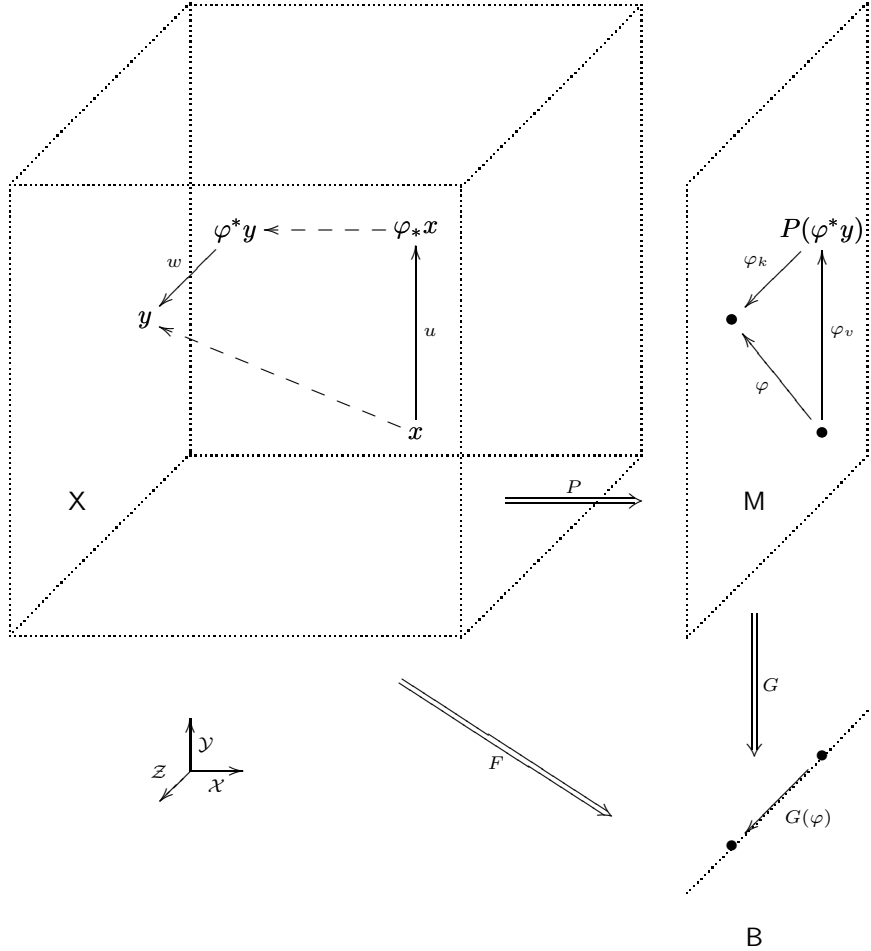
$$
\begin{aligned}
\xy
%
%
(24,0)*{};(0,-24)*{} **\dir{.};(60,-24)*{} **\dir{.};(84,0)*{} **\dir{.};(24,0)*{} **\dir{.};
(0,-24)*{};(0,-84)*{} **\dir{.};(60,-84)*{} **\dir{.};(60,-24)*{} **\dir{.};
(60,-84)*{};(84,-60)*{} **\dir{.};(84,0)*{} **\dir{.};
(24,0)*{};(24,-60)*{} **\dir{.};(84,-60)*{} **\dir{.};
(24,-60)*{};(0,-84)*{} **\dir{.};
(9,-66)*{\mathsf{X}};
{\ar (54,-57)*+{x};(54,-30)*+{\varphi_*x}_{u}};
{\ar@{-->} (54,-30)*+{\varphi_*x};(30,-30)*+{\varphi^*y}};
{\ar (30,-30)*+{\varphi^*y};(18,-42)*+{y}_{w}};
{\ar@{-->} (54,-57)*+{x};(18,-42)*+{y}};
{\ar@2{->} (66,-66)*{}; (84,-66)*{}^{P}  };
%
%
(114,0)*{};(90,-24)*{} **\dir{.};(90,-84)*{} **\dir{.};(114,-60)*{} **\dir{.};(114,0)*{} **\dir{.};
(99,-66)*{\mathsf{M}};
{\ar (108,-57)*+{\bullet}; (108,-30)*+{P(\varphi^*y)}_{\varphi_v}};
{\ar (108,-57)*+{\bullet}; (96,-42)*+{\bullet}^{\varphi}};
{\ar (108,-30)*+{P(\varphi^*y)}; (96,-42)*+{\bullet}_{\varphi_k}};
{\ar@2{->} (99,-81)*{}; (99,-100)*{}^{G}  };
%
%
(90,-118)*{};(114,-94)*{} **\dir{.};
(99,-124)*{\mathsf{B}};
{\ar (108,-100)*+{\bullet};(96,-112)*+{\bullet}^{G(\varphi)}};
{\ar@2{->} (52,-90)*{}; (80,-108)*{}_{F}  };
%
%
{\ar (24,-102)*{}; (31,-102)*{}_{\mathcal X}};
{\ar (24,-102)*{}; (24,-95)*{}_{\mathcal Y}};
{\ar (24,-102)*{}; (20,-106)*{}_{\mathcal Z}};
\endxy
\end{aligned}
$$
The reader may find it useful to follow the steps of the constructions involved in the theorem above on the diagram provided. Here, the fibres of $F$ are represented by $\mathcal{XY}$-square sections in $\cX$, the fibres of $G$ by $\mathcal Y$-line segments in $\cM$ and the fibres of $P$ by $\mathcal X$-line segments in $\cX$.
\end{proof}

\medskip
So far we have established a bijection $\Phi$ between the set of morphisms we were to describe and \emph{another} set of morphisms. This is not a big deal, unless we have some information on the second set of morphisms. It turns out that this can be the case, when we start with a fibrewise opfibration  $(P,F,G)$ such that the fibres of $P$ are groupoids. In this case, since the set $\cX_{P(\varphi^*y)}(\varphi_*x,\varphi^*y)$ lives in the fibre of $P$ over $P(\varphi^*y)$, we can use the structure of the fibres of $P$ to classify the maps in $\cX$. 

\begin{Theorem} \label{thm:structure}
Let us consider a fibrewise opfibration  $(P,F,G)$ such that the fibres of $P$ are groupoids.
Given $x$ and $y$ objects of $\cX$, and $\varphi\colon P(x)\to P(y)$, we  consider the set
$\cX_{\varphi}(x,y)$. Then
\begin{itemize}
\item[\emph{(i)}]  $\cX_{\varphi}(x,y)\neq \emptyset$ if and only if $\varphi^*y\cong \varphi_*x$ in the $P$-fibre.
\item[\emph{(ii)}] If this is the case, then  $\cX_{\varphi}(x,y)$ is a simply transitive $H$-set (i.e.\ an $H$-torsor), where the acting group is
$$ H=\cX_{P(\varphi^*y)}(\varphi^*y,\varphi^*y).$$
\end{itemize}
\end{Theorem}

\begin{proof}
Use $\Phi$ to transfer the simply transitive left action of $H$ on
$$\cX_{P(\varphi^*y)}(\varphi_*x,\varphi^*y)$$
given by arrow composition.
\end{proof}
\begin{Remark}
Of course, one can equivalently endow the set $\cX_{P(\varphi^*y)}(\varphi_*x,\varphi^*y)$ with a simply transitive right action
of the group
$$\cX_{P(\varphi_*x)}(\varphi_*x,\varphi_*x).$$
\end{Remark}
When a fibrewise opfibration  $(P,F,G)$ satisfies the hypothesis of Theorem \ref{thm:structure}, we say that it admits a \emph{categorical obstruction theory}. In this case, that theorem provides a solution to the obstruction problem of Definition \ref{def:obstruction_problem}, for any choice of $(x,y,\varphi)$, by describing explicitly the torsor structure of the sets of maps $f\colon x \to y$ such that $P(f)=\varphi$.

\section{Applications - Part I}

In this section we describe three direct applications of Theorem \ref{thm:structure}.

\subsection{Group extensions with abelian kernel} \label{sec:OPEXT}

Here we apply Theorem \ref{thm:structure} to recover the classification of group extensions (i.e.\ short exact sequences) with abelian kernel.  The same method can be applied to extensions defined in any semi-abelian category: the interested reader can consult \cite{pf}.
 
Consider the diagram of categories and functors:
\begin{equation}\label{diag:ab_ext}
\begin{aligned}
\xymatrix{
\OPEXT(\Gp)\ar[rr]^{P}\ar[dr]_{P_0}
&&\Mod(\Gp)\ar[dl]^{(\ )_0}
\\
&\Gp}
\end{aligned}
\end{equation}
where  $\OPEXT(\Gp)$ is the category of group extensions with abelian kernel and their morphisms, $P$ is the  functor that assigns to any $C$-extension with abelian kernel $B$ the induced $C$-module structure on $B$,  $(\ )_0$ is the forgetful functor that sends a $C$-module to the acting group $C$. Diagram (\ref{diag:ab_ext}) is a fibrewise opfibration, where the $(\ )_0$-cartesian liftings are given by precomposition, the $P_0$-cartesian liftings are given by pullback and the opcartesian liftings in the restriction to fibres are given by push forward (see \cite{Homology}). Let us fix a $C$-module $B$, with action $\xi\colon C\times B\to B$; the $P$-fibre $\OPEXT(\Gp)(C,B,\xi)$ over the $C$-module $B$ is the category of $C$-extensions of $B$ that induce $\xi$. Let us notice that, by the short-five lemma, $P$-fibres are groupoids. 

\smallskip
\noindent\emph{Obstruction problem.} Let us consider two extensions with abelian kernel 
$$E=(f,k)\quad \text{and}\quad E'=(f',k')$$ and a morphism between the induced modules: $(\varphi_0,\varphi_1)\colon \xi\to \xi'$. 
\begin{equation} \label{diag:opext_to_mod}
P\colon\quad
\begin{aligned}
\xymatrix{B\ar[r]^{k}\ar[d]_{\varphi_1}&E\ar[r]^{f} \ar@{.>}[d]|(.4){?}&C\ar[d]^{\varphi_0}
\\
B'\ar[r]_{k'}&E'\ar[r]_{f'} &C'}
\end{aligned}
\qquad\mapsto\qquad
\begin{aligned}
\xymatrix{C\times B\ar[r]^-{\xi}\ar[d]_{\varphi_0\times \varphi_1}&B\ar[d]^{\varphi_1}
\\
C'\times B'\ar[r]_-{\xi'}&B'}
\end{aligned}
\end{equation}
The obstruction  problem in this case, consists in  determining whether there are morphisms of extensions $E\to E'$ which induce $(\varphi_0,\varphi_1)$. According to point (i) of Theorem \ref{thm:structure}, this is the case if and only if the extension $\varphi_0^*E'$ is isomorphic to the extension $\varphi_{1*}E$ in the $P$-fibre over $(C,B',\varphi_0^*\xi')$, where the $C$-module structure $\varphi_0^*\xi'$ is given by pulling back the action $\xi'$ along $\varphi_0$. The situation is represented by the diagram below
\begin{equation}
\begin{aligned}
\xymatrix{B\ar[r]^{k}\ar[d]_{\varphi_1}&E\ar[r]^{f} \ar[d] &C\ar@{=}[d]
\\
B'\ar@{=}[d]\ar[r]&\varphi_{1*}E\ar[r]\ar@{.>}[d]|(.4){?} &C\ar@{=}[d]\\
B'\ar@{=}[d]\ar[r]&\varphi_0^*E'\ar[r]\ar[d]&C\ar[d]^{\varphi_0}\\
B'\ar[r]_{k'}&E'\ar[r]_{f'} &C'}
\end{aligned}
\end{equation}
where lower right square is a pullback, while the upper left one is a push forward (compare with Corollary 6.7 in \cite{pf} for the semi-abelian case).

\smallskip
\noindent\emph{Classification}. According to point (ii) of Theorem \ref{thm:structure}, when it is not empty, the set of extensions related with the obstruction problem stated above is a simply transitive $Z$-set, where $Z$ is the group of automorphisms of the extension $\varphi_0^*E'$, i.e.\ of the group isomorphisms  $\varphi_0^*E'\to\varphi_0^*E'$ that fix the kernel $B'$ and the cokernel $C$. One can prove that $Z$ is isomorphic  to the classical group $\mathcal Z_{\varphi_0^*\xi'}^1(C,B')$ of \emph{crossed homomorphism} $C\to B'$ (see for instance \cite[Chapter IV]{Homology}).

\begin{Remark}\label{rem:cat_group}
The isomorphism $Z\simeq \mathcal Z_{\varphi_0^*\xi'}^1(C,B')$ can be proved directly. Otherwise, one can see this isomorphism as a natural consequence of a more sophisticated general argument, which will be only outlined here.
As it was observed in \cite{Vitale03} (see also Section 6.2 of \cite{CM16}), given a $C$-module $B$ with action $\xi$, the Baer sum endows the groupoid $\OPEXT(\Gp)(C,B,\xi)$ with a natural symmetric monoidal structure which makes it a \emph{categorical group} (\cite{Sinh}, see also Section \ref{sec:catgp}), since every object is invertible up to isomorphism; the identity object is the canonical semidirect product extension determined by $\xi$. In fact, one can prove that the group $\pi_0(\OPEXT(\Gp)(C,B,\xi))$ of the connected components of this categorical group is isomorphic to $\mathcal H^2(C,B,\xi)$, and that the abelian group $\pi_1(\OPEXT(\Gp)(C,B,\xi))$ of the automorphisms of the identity object is in fact isomorphic to $\mathcal Z^1(C,B,\xi)$. From basic \emph{categorical group} theory (see for example \cite{GI}),  the automorphism group of the identity object is naturally isomorphic to the automorphism group of \emph{any}  other object. This gives the desired isomorphism in the situation considered above.
\end{Remark}

To sum up, we state the following result ($(i)$ was already present in \cite{pf} as Corollary 5.7).

\begin{Theorem} \label{thm:classification_gp}
Let us consider two group extensions with abelian kernel 
$$E=(f,k)\quad \text{and}\quad E'=(f',k')$$ and a morphism between the induced modules: $(\varphi_0,\varphi_1)\colon \xi\to \xi'$ (see diagram \eqref{diag:opext_to_mod} above). Then 
\begin{itemize}
\item[\emph{(i)}] There exist morphisms of extensions $E\to E'$ which induce $(\varphi_0,\varphi_1)$ if and only if $\varphi_0^*E'\cong\varphi_{1*}E$.
\item[\emph{(ii)}] In this case, $\OPEXT_{(\varphi_0,\varphi_1)}(E,E')$ is a simply transitive $\mathcal Z^1(C,B',\varphi_0^*\xi')$-set.
\end{itemize}
\end{Theorem}

\begin{Remark}
If $E$ determines an element $[\epsilon]$ of $\mathcal H^2(C,B)$ represented by a cocycle $\epsilon\colon C\times C \to B$, then ${\varphi_1}_*E$ corresponds to the element $[\varphi_1\cdot \epsilon]$ of $\mathcal H^2(C,B')$. On the other side, if $E'$ determines an element $[\epsilon']$ of $\mathcal H^2(C',B')$, then ${\varphi_0}^*E'$ corresponds to the element $[\epsilon'\cdot(\varphi_0\times\varphi_0)]$ of $\mathcal H^2(C,B')$. Hence, point (i) may be rephrased as follows:
\begin{itemize}
\item[(i')] \emph{There exist morphisms of extensions $E\to E'$ which induce $(\varphi_0,\varphi_1)$ if and only if $[\varphi_1\cdot\epsilon]=[\epsilon'\cdot(\varphi_0\times\varphi_0)]$, or, equivalently, if the obstruction $\varphi_1\cdot\epsilon-\epsilon'\cdot(\varphi_0\times\varphi_0)$ is cohomologous to zero.}
\end{itemize}
\end{Remark}

\subsection{A comparison with Bourn's direction functor}

The \emph{direction functor approach} to cohomology introduced by Bourn in \cite{Bourn99} (see also \cite{Bourn99b,Bourn08,BR07}) can be used to provide examples of fibrewise opfibrations and related obstruction theories. Here we briefly recall the notion of direction functor in low dimension, and then we describe the way it can be related to our theory. 

Let $\cE$ be a Barr-exact category, and $\cE_g$ its full subcategory determined by the objects $X$ with global support (i.e.\ such that the terminal map $X\to 1$ is a regular epimorphism). Moreover, let us denote by $\mathsf{AM}(\cE)$ the category of associative Mal'tsev operations 
$$
p\colon X\times X\times X\to X
$$
in $\cE$, and by $\mathsf{AutM}(\cE)$ the subcategory of autonomous Mal'tsev operations in $\cE$. Bourn defined a \emph{direction} functor $d_{\cE} \colon \AM(\cE_g)\to \Gp(\cE)$, and studied several important properties of this functor. Such $d_{\cE}$ is not an opfibration, but just a pseudo-opfibration with groupoidal fibres, as showed in \cite{Bourn99}. Furthermore its restriction to $\AutM(\cE_g)$ factors through $\Ab(\cE)$.

In fact, the functors $d_{\cE}$ are the components of a 2-natural transformation $d$: 
$$
\xymatrix@C=10ex{
\mathsf{ExCat}\ar@/^3ex/[r]^{\AM(\ \ {}_g)}_{}="1"
\ar@/_3ex/[r]_{\Gp(\ )}^{}="2"
&\Cat
\ar@{=>}"1";"2"_{d}
}
$$
Now, via the Grothendieck construction (see, for example, \cite[B1.3]{elephant}), it is clear that any contravariant pseudo-functor $\cB^\mathsf{op}\to\mathsf{ExCat}$ yields, by composition with $d$, a morphism in $\Fib(\cB)$ whose restrictions to the fibres are the above mentioned direction functors. In particular, for a Barr-exact category $\cB$, we can consider the pseudo-functor that assigns to each object $C$ of $\cB$, the slice category $\cB/C$, i.e.~the one corresponding to the fundamental fibration $\mathrm{cod}\colon\mathsf{Arr}(\cB)\to\cB$. In this way, one obtains the morphism in $\Fib(\cB)$ represented below, whose restrictions to the fibres over $\cB$ are pseudo-opfibrations:
\begin{equation}\label{fibrewise_direction}
\begin{aligned}
\xymatrix{
\displaystyle\int_{C}\AM(\cB/C)_g\ar[rr]^{P}\ar[dr]_{P_0}
&&\displaystyle\int_{C}\Gp(\cB/C)\ar[dl]^{(\ )_0}
\\
&\cB}
\end{aligned}
\end{equation}
and the corresponding restriction to the autonomous case:
\begin{equation}\label{fibrewise_ab_direction}
\begin{aligned}
\xymatrix{
\displaystyle\int_{C}\AutM(\cB/C)_g\ar[rr]^{P}\ar[dr]_{P_0}
&&\displaystyle\int_{C}\Ab(\cB/C)\ar[dl]^{(\ )_0}
\\
&\cB}
\end{aligned}
\end{equation}
The category $\Ab(\cB/C)$ is known as the category of Beck $C$-modules, and $\int_{C}\Ab(\cB/C)$ is also called the \emph{tangent category} of $\cB$.

If the category $\cB$ is not only Barr-exact, but also Mal'tsev, then groups in $\cB/C$ are automatically abelian, and Mal'tsev operations in $\cB/C$ are automatically autonomous, so that \eqref{fibrewise_direction} and \eqref{fibrewise_ab_direction} coincide.

In many interesting cases, as for example in the context of a semi-abelian category $\cB$, we can represent the latter diagram \eqref{fibrewise_ab_direction} up to equivalences in $\Fib(\cB)$ by means of a fibrewise (genuine) opfibration. This is displayed in the following diagram:
\begin{equation} \label{diag:opext_semiab}
\begin{aligned}
\xymatrix{
\OPEXT(\cB)\ar[rr]^{P}\ar[dr]_{P_0}
&&\Mod(\cB)\ar[dl]^{(\ )_0}
\\
&\cB}
\end{aligned}
\end{equation}
where $\OPEXT(\cB)$ has to be considered as the category of abelian extensions (which are, in general, a subcategory of extensions with abelian kernel, see \cite{BJ}) and $\Mod(\cB)$ as the category of abelian actions, i.e.\ internal actions (see \cite{BJK}) associated with Beck modules in \cB. In fact, diagram \eqref{diag:ab_ext} is nothing but the specification of diagram \eqref{diag:opext_semiab} in \Gp. On the other hand, Theorem \ref{thm:classification_gp} admits a generalization to the semi-abelian context (for all the notions involved, the reader may refer to \cite{CM16}).

\begin{Theorem} \label{thm:opext_semiab}
In a semi-abelian category \cB, let us consider two abelian extensions 
$$E=(f,k)\quad \text{and}\quad E'=(f',k')$$ and a morphism between the induced abelian actions: $(\varphi_0,\varphi_1)\colon \xi\to \xi'$. Then 
\begin{itemize}
\item[\emph{(i)}] There exist morphisms of extensions $E\to E'$ which induce $(\varphi_0,\varphi_1)$ if and only if $\varphi_0^*E'\cong\varphi_{1*}E$.
\item[\emph{(ii)}] In this case, $\OPEXT(\cB)_{(\varphi_0,\varphi_1)}(E,E')$ is a simply transitive\\ $\pi_1(\OPEXT(\cB)(C,B',\varphi_0^*\xi'))$-set.
\end{itemize}
\end{Theorem}

Thus we have made explicit the link between low dimensional Bourn cohomology and our approach to categorical obstruction theory. The same can be done in higher dimensions, starting from crossed extensions, even if the fibres of $\Pi$ are not groupoids. This issue is dealt with in Sections \ref{sec:Loc} and \ref{sec:appl2}.

\subsection{Singular extensions of unital associative algebras}

Diagram \eqref{diag:opext_semiab} can be still representative also outside the semi-abelian context, as in the case $\AssAlg_1$ of unital associative algebras over a field $\mathbb{K}$, provided we interpret $\OPEXT(\AssAlg_1)$ as the category of \emph{singular extensions} and $\Mod(\AssAlg_1)$ as the category \Bimod\ of \emph{bimodules} (see \cite[X.3]{Homology}). Recall that a \emph{bimodule} $B$ over $C$ is a $\mathbb{K}$-vector space endowed with left and right unital $C$-actions such that for $c, c' \in C$ and $b \in  B$ we have $(c*b)*c' = c*(b*c')$.  Morphisms of bimodules are obviously defined.

As in the case of groups, given a $C$-bimodule $B$, from Corollary 9 in \cite{Bourn99} we deduce that the groupoid $\OPEXT(\AssAlg_1)(C,B)$ is endowed with a symmetric monoidal closed structure which makes it a symmetric categorical group. In fact, following \cite{Homology}, we can recover the second \emph{Hochschild cohomology} group $\mathcal H^2_{_H}(C,B)$ as the group $\pi_0(\OPEXT(\AssAlg_1)(C,B))$ of isomorphism classes of 
this categorical group. 

The abelian group $\pi_1(\OPEXT(\AssAlg_1)(C,B))$ of the automorphisms of the identity object (given by the semi-direct sum \cite{Homology}) can be equivalently represented by the group of \emph{crossed homomorphisms} (see \cite[X.(3.4)]{Homology}), so that it is isomorphic to the group $\mathcal Z^1_{_H}(C,B)$. Since we are dealing with a categorical group, the automorphism group of the identity object is naturally isomorphic to the automorphism group of \emph{any} other object. Hence we get an analog of Theorem \ref{thm:classification_gp}.

\begin{Theorem} \label{thm:classification_aa}
Let us consider two singular extensions of unital associative algebras 
$$E=(f,k)\quad \text{and}\quad E'=(f',k')$$ and a morphism $(\varphi_0,\varphi_1)\colon (C,B)\to (C',B')$ between the induced bimodules. Then 
\begin{itemize}
\item[\emph{(i)}] There exist morphisms of extensions $E\to E'$ which induce $(\varphi_0,\varphi_1)$ if and only if $\varphi_0^*E'\cong\varphi_{1*}E$.
\item[\emph{(ii)}] In this case, $\OPEXT_{(\varphi_0,\varphi_1)}(E,E')$ is a simply transitive $\mathcal Z^1_{_H}(C,B')$-set (where the $C$-bimodule structure on $B'$ is induced by the one of $C'$ via $\varphi_0$).
\end{itemize}
\end{Theorem}

\section{Localization of a fibrewise opfibration}\label{sec:Loc}

In the first part of this section, we introduce a prototype example that concerns the setting described in Section \ref{sec:setting}. Since in this example the assumption that the fibres of $P$ are groupoids is not satisfied, in order to apply Theorem \ref{thm:structure} we have to perform a suitable localization. This is explained in the second part of this section.

\subsection{Crossed modules and crossed extensions of groups}\label{subsec:Xmod}

Here we only recall the basic notions, which are well established in the literature; we refer to \cite{CMMV_Yoneda} for some constructions and results. The whole section is based on the category of groups, but the notions and results therein can be adapted to many semi-abelian categories. For this reason, we shall often replace group-theoretical terms by categorical ones, as for example using ``regular epimorphism'' instead of ``surjection''.

\begin{Definition}
A pre-crossed module is a group homomorphism 
$\partial\colon G_2\to G_1$ endowed with an action of $G_1$ on $G_2$ such that, for every $g_1\in G_1$ and $g_2 \in G_2$, the condition
$$
(i)\ \ \partial(g_1*g_2)=g_1\partial(g_2)g_1^{-1}\,,
$$
is satisfied. If moreover, for every  $g_2,g_2' \in G_2$, the condition
$$
(ii)\ \ \partial(g_2)*g_2'=g_2+g_2'-g_2\,,
$$
holds, then $\partial$ is called a \emph{crossed module}.
\end{Definition}
Given two crossed modules $\partial$ and $\partial'\colon G_2'\to G_1'$, a morphism $(f_1,f_2)\colon \partial\to\partial'$ is a pair of group homomorphisms $f_2\colon G_2\to G_2'$ and $f_1\colon G_1\to G_1'$ equivariant with respect to the actions and such that $\partial'\cdot f_2 =f_1\cdot \partial$. The category $\XMod(\Gp)$ of crossed modules is defined. 

Given a crossed module $\partial$, the action of $G_1$ on $G_2$ induces an action of the cokernel $\pi_0(\partial)$ on the abelian kernel $\pi_1(\partial)$ of $\partial$: for $a\in \pi_1(\partial)$ and $x=p(g_1)\in \pi_0(\partial)$, one sets $$
x*a=g_1*j(a)
$$
where $j\colon \pi_1(\partial)\to G_2$ and $p\colon G_1 \to \pi_0(\partial)$ are the kernel inclusion and cokernel quotient respectively. This construction defines a functor $\pi$ in the category $\Mod(\Gp)$ of group modules, i.e.\ of triples $(C,B,\xi)$, where $B$ is a $C$-module via the action $\xi$, and equivariant pairs of morphisms as arrows.

In fact, crossed modules form a $2$-category: a $2$-cell 
$$\alpha\colon (f_1,f_2)\Rightarrow (f'_1,f'_2)\colon \partial\to\partial'$$
is given by a set-theoretical map $\alpha\colon G_1\to G_2'$ satisfying suitable conditions (see, for instance, \cite{AN09}). All $2$-cells are isomorphisms. We shall denote the $2$-category of crossed modules by $\underline{\XMod}(\Gp)$.

A morphism of crossed modules $(f_1,f_2)$ is called \emph{weak equivalence} if $\pi(f_1,f_2)$ is an isomorphism. It is possible to show that every internal equivalence in the $2$-category of crossed modules is a weak equivalence according to the above definition. The converse does not hold in general.

We will need some factorization properties of crossed modules morphisms. Here we  recall the \emph{comprehensive factorization system}; details can be found in \cite[Section 3]{CM16}, where the results are stated in the more general case of a semi-abelian category.
\begin{Proposition}\label{prop:comprehensive}
The category $\XMod(\Gp)$ admits a factorization system whose orthogonal classes are the class of final morphisms and the class of discrete fibrations. These classes are characterized as follows:
\begin{itemize}
\item a morphism $(f_1,f_2)$ is final if and only if $\pi_0(f_1,f_2)$ is an isomorphism and $\pi_1(f_1,f_2) $ is a regular epimorphism;
\item a morphism $(f_1,f_2)$ is a discrete fibration if and only if $f_2$ is an isomorphism.
\end{itemize}
\end{Proposition}

\begin{Remark}\label{rem:gpd}
It is well known that the category of crossed modules in groups is equivalent to the category of internal groupoids, and that such equivalence extends to a biequivalence between the $2$-category of internal crossed modules and the $2$-category of internal groupoids. 
Under this biequivalence, $2$-cells of crossed modules correspond to internal natural transformations. 
The  biequivalence holds true for crossed modules internal to many other algebraic settings (see \cite{AMMV13} for the semi-abelian case).
The construction of the groupoid associated with a given crossed module (and vice versa) can be easily found in the literature, as for example in the cited references. Given a crossed module $\partial_G\colon G_2\to G_1$, its associated internal groupoid is represented by
$$
\xymatrix{
\ar@(ul,ur)^{i}G_2\rtimes G_1\ar@<+1ex>[r]^-{d}\ar@<-1ex>[r]_-{c}
&G_1\ar[l]|-{e}
}
$$
where $d,c,e$ are domain, codomain and identity arrows, while $i$ is the inverse of the groupoid. Notice that the object of arrows is obtained by a semi-direct product.
\end{Remark}

Going back to crossed modules, let us notice that the functor $\pi\colon\XMod(\Gp)\to\Mod(\Gp)$ is not part of a genuine fibrewise opfibration. However we can get our prototype example, replacing $\XMod(\Gp)$ with the equivalent category of crossed extensions.

\begin{Definition}
A crossed extension of groups is an exact sequence in $\Gp$
\begin{equation}
X\colon\qquad
\xymatrix{
0\ar[r]
& B\ar[r]^j
& G_2\ar[r]^{\partial}
& G_1\ar[r]^p
& C\ar[r]
& 1}
\end{equation}
such that $\partial$ is a crossed module.
\end{Definition}
Recall that the homomorphism underlying a crossed module is always proper, i.e.\ $\partial$ factors as a regular epimorphism  followed by a normal monomorphism, so that every crossed module gives rise to a crossed extension, once kernel and cokernel are chosen. 

A morphism of crossed extensions $(\gamma,f_1,f_2,\beta)\colon X\to X' $ is a morphism of exact sequences
$$
\xymatrix{
X\colon\ar@<-1ex>[d]&
0\ar[r]
& B\ar[r]^j\ar[d]_\beta
& G_2\ar[r]^{\partial}\ar[d]_{f_2}
& G_1\ar[r]^p\ar[d]^{f_1}
& C\ar[r]\ar[d]^{\gamma}
& 1
\\
X'\colon&
0\ar[r]
& B'\ar[r]_{j'}
& G'_2\ar[r]_{\partial'}
& G'_1\ar[r]_{p'}
& C'\ar[r]
& 1
}
$$
such that $(f_1,f_2)$ is a morphism of crossed modules. As a consequence, the pair $(\gamma,\beta)$ is a morphism of group-modules.

The category $\XExt(\Gp)$ of crossed extensions of groups is defined, and the assignment 
$$
(\gamma,f_1,f_2,\beta)\mapsto (\gamma,\beta)
$$
defines a forgetful functor
$$
\Pi\colon\XExt(\Gp)\to \Mod(\Gp)\,.
$$
It is clear that the definition of morphism of crossed extensions is redundant, since in $(\gamma,f_1,f_2,\beta)$ the pair $\gamma,\beta$ is uniquely determined by the pair $f_1,f_2$. We keep the additional data only when they make some computations and definitions more evident. Otherwise, we shall often write $(f_1,f_2)$ for a morphism of crossed extensions, or use the vector notation $\underline f=(f_1,f_2).$ The comprehensive factorization of crossed modules morphisms extends naturally to the category of crossed extensions.

Just like crossed modules, also crossed extensions organize in a $2$-category, $2$-cells of crossed extensions being $2$-cells of the underlying crossed modules. The $2$-category of crossed extensions is denoted by $\underline{\XExt}(\Gp)$.

\begin{Proposition}[Theorem 4.2 in \cite{CMMV_Yoneda}]
The commutative diagram of categories and functors
\begin{equation} \label{diag:xext_triang}
\begin{aligned}
\xymatrix{\XExt(\Gp)\ar[rr]^\Pi\ar[dr]_{\Pi_0}&&\Mod(\Gp)\ar[dl]^{(\ )_0}\\&\Gp}
\end{aligned}
\end{equation}
is a fibrewise opfibration over $\Gp$, where $(\ )_0$ is the forgetful functor that assigns to any $C$-module, its acting group $C$.
\end{Proposition}

For the reader's convenience, we sketch the construction of the cartesian liftings of $\Pi_0$ and opcartesian liftings of the restrictions of $\Pi$ to the fibres.

Given a crossed extension $X'$, and a group homomorphism $\gamma \colon C\to C'$, its cartesian lifting  $\hat\gamma\colon \gamma^*X'\to X'$  is given by the following construction:
$$
\xymatrix{
\gamma^*X'\colon\ar@<-1ex>[d]_{\hat\gamma} & B' \ar[d]_1 \ar[r]^{j'} & G'_2 \ar[r]^-{\langle\partial',0\rangle}\ar[d]_1 & G_1'\times_{C'} C \ar[r]^-{pr_2} \ar[d]_{pr_1} & C\ar[d]^{\gamma} \\
X'\colon & B' \ar[r]_{j'} & G'_2\ar[r]_{\partial'} & G_1'\ar[r]_{p'} & C'
}
$$
where the rightmost square is a pullback in $\Gp$. Then the comparison $\langle\partial',0\rangle$ inherits a crossed module structure and $(\gamma,pr_1,1,1)$ is a morphism of crossed extensions, cartesian with respect to
$\Pi_0$.

Now, let us consider a group $C$, a crossed extension
\[
X\colon\qquad
\xymatrix{
0\ar[r]
& B\ar[r]^j
& G_2\ar[r]^{\partial}
& G_1\ar[r]^p
& C\ar[r]
& 1}
\]
and a morphism $\beta\colon \Pi(X)=(B,\phi)\to (B',\phi')$ of $C$-modules.

The strategy to produce an opcartesian lifting of $\beta$ at $X$ is not dual to the one we used for cartesian liftings, as one could argue. Indeed we do not use a pushout, but a push forward  (see \cite{pf}). Let us consider the pair $(\rho,B'\times^BG_2)$, defined by the following short exact sequence of groups:
$$
\xymatrix@C=6ex{
0\ar[r]
&B\ar[r]^-{\langle \beta,-j\rangle}
&B'\times G_2\ar[r]^-\rho
&B'\times^BG_2\ar[r]
&1}
$$
where the normal monomorphism $\langle \beta,-j\rangle$ is given by the assignment
$$
b\mapsto (\beta(b),-j(b))\,.
$$
The following diagram displays the opcartesian lifting:
$$
\xymatrix{
X \colon \ar[d]_{\hat\beta} & B \ar[r]^j \ar[d]_{\beta} & G_2 \ar[r]^{\partial} \ar[d]^{\rho\cdot\langle0,1\rangle} & G_1 \ar[r]^{p} \ar[d]^{1} & C \ar[d]^1
\\
\beta_*X \colon & B' \ar[r]_-{\rho\cdot\langle1,0\rangle} & B'\times^{B} G_2 \ar[r]_-{\delta} & G_1 \ar[r]_{p'} & C
}
$$
where the \emph{push forward} of $j$ along $\beta$ is the normal monomorphism $\rho\cdot\langle1,0\rangle$, which has the same cokernel as $j$, so that we can obtain a homomorphism $\delta$ such that $\delta\cdot\rho\cdot\langle0,1\rangle=\partial$. Moreover, $\delta$ is a crossed module, with action induced from both the $C$-module structure of $B'$ and the crossed module structure of $\partial$.

\subsection{How to get groupoidal fibres}

If a fibrewise opfibration $(P,F,G)$ is such that the fibres of $P$ are not groupoids, Theorem \ref{thm:structure} cannot be applied. This happens, for example, in the case of the functor $\Pi \colon \XExt(\Gp) \to \Mod(\Gp)$ of diagram (\ref{diag:xext_triang}). However, one may try to turn such fibres into groupoids. The idea is to make all the arrows in the fibres of $P$ invertible. This problem has been solved, under suitable conditions, as a consequence of \cite[Proposition 4.8 and Theorem 4.12]{refl_fib}, which we adapt to our context.

\begin{Proposition}\label{prop:factorization}
Let $P\colon F\to G$ be a fibrewise opfibration in $\Fib(\cB)$. If the category $Q$ of fractions of $\cX$ with respect to the class of $P$-vertical arrows is locally small, then $P$ admits a universal factorization in $\Fib(\cB)$
$$
\xymatrix@C=6ex{
\cX\ar[dr]_{F}\ar[r]_{Q}\ar@/^3ex/[rr]^P
&\cQ\ar[d]_{H}\ar[r]_S
&\cM\ar[dl]^G
\\
&\cB
}
$$
through a fibrewise opfibration $S$ whose fibres are groupoids.
\end{Proposition}

\begin{Remark} \label{rem:invertee_identee}
Since $P$ is a fibrewise opfibration in $\Fib(\cB)$, it is an isofibration in $\Cat/\cB$ (see Propositions 2.5 and 2.7 in \cite{CMMV_Yoneda}). Hence, thanks to Corollary 3.3 in \cite{refl_fib}, $Q$ gives also the category of fractions with respect to the class of all the arrows inverted by $P$.
\end{Remark}

By this last statement, the fibrewise opfibration $(S,H,G)$ satisfies the hypothesis of Theorem \ref{thm:structure}, hence it admits a categorical obstruction theory. 
The question that naturally arises at this point concerns the relationship between the obstruction problems set in $(P,F,G)$ and the ones set in $(S,H,G)$. In fact, a canonical construction for the comparison $Q$ is given in \cite{GZ67}, where the (possibly large) category $\cQ$ has the same objects as $\cX$, and arrows given by (equivalence classes of) zig-zag's of arrows, such that all the left directed ones lay in the fibres of $P$. The functor $Q$ is constant on objects, and it sends the morphism $f\colon x\to y$ to itself, seen as a zig-zag.

Then, given $(x,y,\varphi)$ in $(P,F,G)$, we can consider the obstruction problem associated with the same triple, seen in $(S,H,G)$, and try to solve it, i.e.\ determine the set $\cQ_\varphi(x,y)$ of \emph{weak} maps $\xymatrix@!C=2ex{g\colon x \ar[r]|-@{|} & y}$ such that $S(g)=\varphi$. 

\section{Applications - Part II} \label{sec:appl2}

In this section we develop further the example of crossed extensions introduced in Section \ref{subsec:Xmod}. We focus on the case of groups first, even if most of the results hold for a wide class of semi-abelian categories, as for instance Lie algebras, associative algebras, rings, etc (see Section \ref{sec:xext_semiab}). 

\subsection{Crossed extensions of groups} \label{sec:XExt}
Let us consider the fibrewise opfibration represented in diagram (\ref{diag:xext_triang}), and fix a $C$-module $B$. It is well known that the connected components of the fibre of $\Pi$ over $(C,B)$ give an interpretation of the cohomology group $\mathcal H^3(C,B)$, the group operation being defined by means of the Baer sum (see \cite{Brown}). 

The morphisms in the fibre of $\Pi$ over $(C,B)$ are of the kind $(1,f_1,f_2,1)$:
$$
\xymatrix{
0\ar[r]
& B\ar[r]^{j}\ar[d]_{1}
& G_2\ar[r]^{\partial}\ar[d]_{f_2}
& G_1\ar[r]^{p}\ar[d]^{f_1}
& C\ar[d]^{1}\ar[r]
& 1
\\
0\ar[r]
& B\ar[r]_{j'}
& G_2'\ar[r]_{\partial'}
& G_1'\ar[r]_{p'}
& C\ar[r]
& 1
}
$$
This implies that $(f_1,f_2)$ is a weak equivalence of crossed modules. Conversely, any weak equivalence of crossed modules extends to a weak equivalence of crossed extensions lying in a suitable fibre of $\Pi$. 

Weak equivalences do not have inverses in general. The categorical construction that we need in order to make such maps invertible consists in taking the category of fractions with respect to weak equivalences (see Remark \ref{rem:invertee_identee}).

Unfortunately, the class of weak equivalences of crossed extensions does not admit a calculus of fractions (in the sense of \cite{GZ67}), which would allow us to construct the corresponding category of fractions in an easier way. On the other hand, the $2$-category of crossed modules does admit a bicategorical calculus of fractions (in the sense of \cite{Pronk96}) with respect to weak equivalences. This is proved  in \cite{AMMV13} (see also \cite{AN09}) for the more general case of internal crossed modules in a semi-abelian category, where the bicategory of fractions of crossed modules with respect to weak equivalences has a description with \emph{butterflies} as 1-cells (see next section). We are going to use these results in order to show that the classifying category $[\BExt](\Gp)$ of the bicategory  of butterflies (extended to crossed extensions) is the \emph{category of fractions} of crossed extensions with respect to weak equivalences. Since the morphisms inverted by the functor $\Pi$ in diagram (\ref{diag:xext_triang}) are precisely the weak equivalences, thanks to Proposition \ref{prop:factorization} we will then obtain a triangle diagram which still gives a fibrewise opfibration, but with groupoidal fibres (see Corollary \ref{cor:Butter_fac}).

\subsection{The bicategory of butterflies}
Butterflies between crossed modules of groups have been introduced by Noohi in \cite{No08} (see also \cite{AN09}), and they have been extended to crossed extensions in \cite{CM16}, where the more general case of a semi-abelian category is studied. We will mainly refer to the latter, specializing notation and results to the case of groups.

\begin{Definition} \label{def:butterfly}
Let us consider two crossed modules $\partial_H$ and $\partial_G$. A {\em butterfly} $\widehat{E} \colon \partial_H \to \partial_G$ is a commutative diagram in $\Gp$ of the form
\begin{equation} \label{diag:butterfly}
\begin{aligned}
\xymatrix@!=2ex{
	H_2 \ar[rd]^{\kappa} \ar[dd]_{\partial_{H}} & & G_2 \ar[dd]^{\partial_{G}} \ar[ld]_{\iota} \\
	& E \ar[ld]^{\delta} \ar[rd]_{\gamma} \\
	H_1 & & G_1
}
\end{aligned}
\end{equation}
satisfying
\begin{enumerate}
	\item[i.] $(\gamma,\kappa)$ is a complex, i.e.\ $\gamma\cdot \kappa = 0$,
	\item[ii.] $(\delta,\iota)$ is a short exact sequence, i.e.\ $\delta=\coker \iota$ and $\iota=\ker\delta$,
	\item[iii.] The action of $E$ on $H_2$, induced by the one of $H_1$ on $H_2$ via $\delta$, makes $\kappa \colon H_2 \to E$ a pre-crossed module,
	\item[iv.] The action of $E$ on $G_2$, induced by the one of $G_1$ on $G_2$ via $\gamma$, makes $\iota \colon G_2 \to E$ a pre-crossed module.
\end{enumerate}
A $2$-cell $\alpha\colon \widehat{E}\Rightarrow \widehat{E}' \colon \partial_H \to \partial_G$ is a group homomorphism  $\alpha \colon E \to E'$ commuting with the $\kappa$'s, the $\iota$'s, the $\delta$'s and the $\gamma$'s. 
\end{Definition}

Notice that all $2$-cells are necessarily isomorphisms. Compositions and identity butterflies are defined (see \cite{AMMV13}) in order to form the bicategory  $\underline{\Bfly}(\Gp)$ of crossed modules and butterflies. 

\smallskip
Before we go on, let us remark that as far as crossed modules can be considered as a normalized version of internal groupoids, butterflies are a normalized version of  \emph{fractors}, a class of internal profunctors introduced in \cite{MMV13}.

\smallskip
The 2-category of crossed modules embeds into the bicategory of butterflies:
\begin{equation} \label{eq:butter}
\begin{aligned}
\mathcal{B}\colon\underline{\XMod}(\Gp)\to \underline{\Bfly}(\Gp) \,.
\end{aligned}
\end{equation}
The homomorphism $\mathcal B$ of bicategories  is the identity on objects; for a morphism of crossed modules $(f_1,f_2)\colon \partial_H  \to \partial_G$, one defines:
$$ \mathcal{B}(f_1,f_2)=\ \raisebox{+10ex}{
\xymatrix@!=4ex{
	\ar[dd]_{\partial_H} H_2 \ar[dr]|{\langle \partial, i\cdot g\cdot f_2 \rangle}
		& & G_2\ar[dl]|{\langle 0,g \rangle} \ar[dd]^{\partial_G} \\
	& E \ar[dl]^{\overline{d}} \ar[dr]_{c\cdot\overline{f}} \\
	H_1 & & G_1
}
} \ , $$
where
$$ \xymatrix{
	E \ar[r]^{\overline{f}} \ar[d]_{\overline{d}} 
	& G_2\rtimes G_1 \ar[d]^{d} \\
	H_1 \ar[r]_{f_1} & G_1
} $$
is a pullback, $d,c,e,i$ are the structure maps of the internal groupoids associated with the given crossed modules, and $g$ is the kernel of the domain map $d$.
Universal property of pullbacks determines $\mathcal B$ on 2-cells.

Indeed, $\overline{d}$ is a split epimorphism. In fact, every butterfly where $(\delta,\iota)$ is a split short exact sequence comes from a morphism of crossed modules. For this reason, such butterflies are called \emph{representable}.

A special kind of butterfly is given by the class of \emph{flippable} butterflies, i.e.\ those butterflies such that also the pair $(\gamma,\kappa)$ is short exact. In \cite{CM16} it is proved that these are indeed the internal equivalences in the bicategory $\underline{\Bfly}(\Gp)$. Recall from \cite{AMMV13} that a butterfly representing a weak equivalence is flippable, so that $\mathcal B$ sends weak equivalences to internal equivalences.

\subsection{Butterfly composition and spans}\label{sec:extend_PI}
Every butterfly induces a span of crossed modules. The
related construction is represented in the diagram below:
\begin{equation} \label{diag:frac_butterfly}
\begin{aligned}
\xymatrix@!C=4ex{
	& H_2\times G_2 \ar[dl]_{p_1} \ar[dr]^{p_2} \ar[dd]^{\kappa \sharp \iota} \\
	H_2 \ar@{-->}[rd]^{\kappa} \ar[dd]_{\partial_{H}} & & G_2 \ar[dd]^{\partial_{G}} \ar@{-->}[ld]_{\iota} \\
	& E \ar[ld]^{\delta} \ar[rd]_{\gamma} \\
	H_1 & & G_1
}
\end{aligned}
\end{equation}
where the crossed module $\kappa \sharp \iota$ is defined by means of the group operation in $E$ (see \cite{AMMV13} for details), and the morphism $(\delta,p_1)$ is a weak equivalence of crossed modules. As a consequence, the definitions of $\Pi_0$ and $\Pi_1$ extend canonically to butterflies:
$$ 
\Pi_0(\widehat{E})= \Pi_0(\gamma,p_2)\cdot(\Pi_0(\delta,p_1))^{-1} \,, \qquad
\Pi_1(\widehat{E})= \Pi_1(\gamma,p_2)\cdot(\Pi_1(\delta,p_1))^{-1} \,. 
$$

The span representation makes it easier to compute compositions of butterflies. In this subsection,  the results are presented in a straightforward way by using the comprehensive factorization of a morphism into a final one followed by a discrete fibration, as described in Proposition \ref{prop:comprehensive}. Proofs and details can be deduced from \cite{M19}, where the case of internal groupoids in Barr-exact categories is dealt with.

\begin{Proposition}
Consider two composable butterflies $\widehat E$ and $\widehat E'$ together with their span representation:
$$
\widehat E=(\underline s, \underline t)\colon \partial_H\to\partial_G\qquad
\widehat E'=(\underline s', \underline t')\colon \partial_G\to\partial_K
$$
Then, their composite $\widehat E'\cdot \widehat E=(\underline s'', \underline t'')$ can be computed as follows: 
\begin{equation}\label{diag:composition}
\begin{aligned}
\xymatrix@!=3ex{
&&\partial_E\times_{\partial_G}\partial_{E'}\ar[dr]|!{[ddll];[rr]}\hole^{\underline v} \ar[dl]_{\underline u} \ar[rr]^-{\underline q}
&& \partial_{E'E} \ar[ddllll]_{\underline s''}\ar[dd]^{\underline t''}
\\
&\partial_E\ar[dr]|!{[rrru];[dl]}\hole_(.6){\underline t} \ar[dl]_{\underline s}&&\partial_{E'}\ar[dr]_{\underline t'}\ar[dl]^{\underline s'}\\
\partial_H&&\partial_G&&\partial_K
}
\end{aligned}
\end{equation}
take  the pullback  $\underline t\cdot \underline u = \underline s' \cdot \underline v$, and then consider the comprehensive factorization  
$$
\langle\underline s\cdot\underline u,\underline t'\cdot\underline v\rangle= \langle \underline s'',\underline t''\rangle\cdot \underline q\colon \partial_E\times_{\partial_G}\partial_{E'}\to \partial_H\times \partial_K\,.
$$
The discrete fibration $\langle \underline s'',\underline t''\rangle$ corresponds to a butterfly which gives the required composition. 
\end{Proposition}
\begin{Remark}\label{rem:we_pb}
In fact, the final morphism $\underline q$ is not just final but it is a weak equivalence. Actually, the spans representing butterflies are rather special ones, with the left leg belonging to the pullback stable class of surjective on objects weak equivalences. This implies that $\underline s\cdot \underline u$ is a weak equivalence, as well as $\underline s''$, and also $\underline q$ is, by the \emph{2 out of 3} property.
\end{Remark}

\subsection{Butterflies and crossed extensions} \label{sec:bfly_gp}
Given the construction in Section \ref{sec:extend_PI}, we can easily extend the notion of butterfly from crossed modules to crossed extensions. As a result, we obtain the bicategory  $\underline{\BExt}(\Gp)$: a butterfly between two crossed extensions is nothing but a butterfly between the underlying crossed modules. Recall from \cite{CM16} the definition of  $[\BExt](\Gp)$, i.e.\ the classifying category of  $\underline{\BExt}(\Gp)$: objects are crossed extensions in $\Gp$ and  arrows are isomorphism classes of butterflies. 

The homomorphism $\mathcal B$ of bicategories (\ref{eq:butter}) sends a morphism of crossed modules to a butterfly, and this assignment preserves associativity and identities up to 2-cells. Hence, a functor is defined:
$$ Q\colon \XExt(\Gp) \to [\BExt](\Gp) $$
which is the identity on objects and takes any morphism $\underline f=(f_1,f_2)$ of crossed extensions to the class $[\mathcal B(U(\underline f))]$ of butterflies. 

Now, given the butterfly $\widehat{E}$ representing a morphism $[\widehat{E}]$ in $[\BExt](\Gp)$, we can associate with it a span 
$$
(\underline s,\underline t)=((\delta,p_1),(\gamma,p_2))
$$
in $\XExt(\Gp)$ obtained by extending the construction of the last section to crossed extensions (see diagram (\ref{diag:frac_butterfly})).

Moreover, since $\underline s$ is a weak equivalence, $\mathcal B(U(\underline s))$ is invertible up to isomorphism, and, according to Theorem 5.6 in \cite{AMMV13}, $\widehat{E}\cdot\mathcal B(U(\underline s))\cong\mathcal B(U(\underline t))$. Then $Q(\underline s)$ is invertible and $Q(\underline t)\cdot Q(\underline s)^{-1}=[\widehat{E}]$. 
Then, we get the following result.
 
\begin{Proposition} \label{prop:fractions}
The functor $Q$ defined above presents $[\BExt](\Gp)$ as the category of fractions of $\XExt(\Gp)$ with respect to weak equivalences.
\end{Proposition}

\begin{proof}
Since the butterfly associated with a weak equivalence is an internal equivalence in $\underline{\BExt}(\Gp)$, it is clear that $Q$ sends weak equivalences to isomorphisms. Moreover, it is universal with respect to this property. 

In order to prove it, let us consider a functor $F\colon \XExt(\Gp)\to \mathsf X$ that turns weak equivalences into isomorphisms.
$$
\xymatrix{
\XExt(\Gp)\ar[r]^Q\ar[dr]_{F}
&[\BExt](\Gp)\ar@{-->}[d]^{\widetilde F}
\\
&\cX}
$$
There is just one way to define the functor $\widetilde F$ which extends $F$. Since $Q$ is constant on objects, $\widetilde F$ behaves like $F$ on objects. As for morphisms, let us consider  $[\widehat E]\colon \partial_H\to \partial_G$, together with a span representation $\widehat E=(\underline s,\underline t)$, and define 
$$
\widetilde F([\widehat E])=F(\underline t)\cdot F(\underline s)^{-1}\,.
$$
It is well-defined, since isomorphic butterflies have isomorphic span representations. For two composable butterflies $\widehat E$ and $\widehat E'$ as in diagram (\ref{diag:composition}), we have:
$$
\widetilde F([\widehat E'])\cdot\widetilde F([\widehat E])=F(\underline t')\cdot F(\underline s')^{-1}\cdot F(\underline t)\cdot F(\underline s)^{-1}=F(\underline t')\cdot F(\underline v)\cdot F(\underline u)^{-1}\cdot F(\underline s)^{-1}
$$
$$
=F(\underline t'\cdot \underline v)\cdot F(\underline s\cdot \underline u)^{-1}
=F(\underline t''\cdot \underline q)\cdot F(\underline s''\cdot \underline q)^{-1}=F(\underline t'')\cdot F( \underline q)\cdot F(\underline q)^{-1}\cdot F(\underline s'')^{-1}
$$
$$
=F(\underline t'')\cdot F(\underline s'')^{-1} = \widetilde F([\widehat E'\cdot\widehat E])= \widetilde F([\widehat E']\cdot[\widehat E])
$$
where the second equality holds because $\underline t\cdot\underline u=\underline s'\cdot \underline v$, and $\underline u$ is a weak equivalence,  the fifth one holds because the final morphism $\underline q$ is in fact a weak equivalence as well (see Remark \ref{rem:we_pb}). From preservation of composition, in this case one gets the  preservation of identities, and this shows that $\widetilde F$ is a functor.
Finally, given a morphism $\underline f\colon \partial_H\to \partial_G$, we are to prove that $\widetilde F(Q(\underline f))= F(\underline f)$. To this end, recall that the morphism $\underline f$ gives rise to a butterfly $\mathcal B(U(\underline f))$ whose representing span has  a \emph{split epimorphic} internal equivalence as left leg:
$$
\xymatrix{
\partial_H
&\partial_f\ar[l]_{\underline s}\ar[r]^{\underline t}
&\partial_G}\,.
$$
Recall from \cite[Remark 5.7]{AMMV13} that, given a section $\underline s'$ of $\underline s$,  $\underline t\cdot \underline s'$ is isomorphic to $\underline f$. In fact, one can choose the section $\underline s'$ in such a way that the isomorphism is an equality. Moreover, since $\underline s$ is a weak equivalence, $F(\underline s)$ is an isomorphism, and $F(\underline s)^{-1}=F(\underline s').$ Then:
$$
\widetilde F(Q(\underline f))=
\widetilde F([\mathcal B(U(\underline f))])=
F(\underline t)\cdot F(\underline s)^{-1}=
F(\underline t)\cdot F(\underline s')=
F(\underline t \cdot \underline s')=F(\underline f)\,.$$
\end{proof}

\begin{Remark}
Another point of view in dealing with the category of fractions of crossed modules with respect to weak equivalences is adopted in \cite{No07}. In particular, see Proposition 9.9 and Proposition 9.12 therein.
\end{Remark}

From Proposition \ref{prop:factorization}, one immediately obtains the following statement.

\begin{Theorem} \label{cor:Butter_fac}
The fibrewise opfibration $(\Pi,\Pi_0,(\ )_0)$ admits a factorization $\Pi=P\cdot Q$ such that the resulting fibrewise opfibration $(P,P_0,(\ )_0)$ has groupoidal fibres.
$$
\xymatrix@C=6ex{
\XExt(\Gp)\ar[dr]_{{\Pi}_0}\ar[r]_{Q}\ar@/^3ex/[rr]^{\Pi}
&[\BExt](\Gp)\ar[d]_{P_0}\ar[r]_P
&\Mod(\Gp)\ar[dl]^{(\ )_0}
\\
&\Gp
}
$$
\end{Theorem}
One can find below a diagram describing how to obtain $P([\widehat E])=(\varphi_0,\varphi)$.
\begin{equation} \label{diag:span_biext}
\begin{aligned}
\xymatrix@!C=4ex{
	B \ar[dd]_j \ar@{=}[r] & B \ar@{-->}[d]_{\langle j,j'\varphi \rangle} \ar@{-->}[r]^\varphi & B' \ar[dd]^{j'} \\
	& H_2\times G_2 \ar@{-->}[dl]_{p_1} \ar@{-->}[dr]^{p_2} \ar@{-->}[dd]^{\kappa \sharp \iota} \\
	H_2 \ar[rd]^{\kappa} \ar[dd]_{\partial_{H}} & & G_2 \ar[dd]^{\partial_{G}} \ar[ld]_{\iota} \\
	& E \ar[ld]^{\delta} \ar[rd]_{\gamma} \ar@{-->}[dd]_{p\delta} \\
	H_1 \ar[d]_p & & G_1 \ar[d]^{p'} \\
	C \ar@{=}[r] & C \ar@{-->}[r]_{\varphi_0} & C' 
}
\end{aligned}
\end{equation}

We can now apply Theorem \ref{thm:structure} to our case study.

\begin{Theorem} \label{thm:XExt}
Given two crossed extensions $X=(p,\partial,j)$ and $X'=(p',\partial',j')$, and a homomorphism of modules $\uffa=(\varphi_0,\varphi)\colon \Pi(X)\to \Pi(X')$, then
\begin{itemize}
\item[i)] the set (of the isomorphism classes) of butterflies
$$[\widehat E]\colon X \to X'$$
with $P[\widehat E]=\uffa$ is not empty if and only if  $\uffa^*X'$ is isomorphic to $\uffa_*X$ in $[\BExt](\Gp)$;
\item[ii)] if this is the case, such a set is a simply transitive $\Gamma$-$\Set$, where  $$
 \Gamma= [\BExt](\Gp)_{\Pi(\uffa^*X')}(\uffa^*X',\uffa^*X')
$$
is the group of the automorphisms of the object $\uffa^*X'$ in the fibre of $[\BExt](\Gp)$ over $\Pi(\uffa^*X')$.
\end{itemize}
\end{Theorem}

The diagram below may help to understand Theorem \ref{thm:XExt}.
\begin{equation}\label{diag:the_example}
\begin{aligned}
\xymatrix@!C=8ex{
	B \ar[d] \ar[r]_{\varphi}  \ar@<+3ex>[rrrr]^{\varphi} & B' \ar[d]_{} \ar[rr]_-1
		& & B' \ar[d]_{} \ar[r]_-{1} & B' \ar[d]^{} 
\\
	H_2\ar[r]\ar[dd]_{\partial}  & H_2'\ar@{-->}[dr]\ar[dd]_{\uffa_*\partial}
		& & G_2\ar@{-->}[dl] \ar[dd]^{\uffa^*\partial'}\ar[r]^{1} & G_2\ar[dd]^{\partial'} 
\\
	&  & ? \ar@{-->}[dl] \ar@{-->}[dr] 
\\
	H_1 \ar[r]_-{1} \ar[d]_{}  & H_1 \ar[d]_{} & & G_1' \ar[d] \ar[r]_{}
		& G_1 \ar[d]
\\
	C \ar[r]^1  \ar@<-3ex>[rrrr]_{\varphi_0} & C \ar[rr]^1 & & C \ar[r]^-{\varphi_0} & C'
}
\end{aligned}
\end{equation}
First, since $(\ )_0$ is a fibration, we consider a cartesian lifting of $\varphi_0$ at $X'$. This is done via the pullback represented by the right lower square of the diagram. Notice that, at this point, since composition with a cartesian map induces an isomorphism between the homsets under consideration, we have reduced our problem to finding a weak map, represented by a butterfly, between $\uffa_*X$ and $\uffa^*X'$.

Now we can consider the restriction of the fibrewise opfibration $\Pi$ to the fibres over $C$, which is then an opfibration. 
The opcartesian lifting of $\varphi$ at $X$ is obtained via the push forward in the left upper square in the diagram. Moreover, since also composition with an opcartesian map induces isomorphisms between the homsets under consideration, we have reduced further our problem to finding a weak map between $\uffa_*X$ and $\uffa^*X'$. 
This is represented by the dashed butterfly in the diagram, that, when it exists, yields an isomorphism in $[\BExt](\Gp)$, since it lays in a $P$-fibre. This explains point $i)$ in Theorem \ref{thm:XExt}, which has also a cohomological interpretation. If $X$ determines an element $[\epsilon]$ of $\mathcal H^3(C,B)$ represented by a cocycle $\epsilon\colon C\times C\times C \to B$, then $\uffa_*X$ corresponds to the element $[\varphi\cdot \epsilon]$ of $\mathcal H^3(C,B')$. On the other side, if $X'$ determines an element $[\epsilon']$ of $\mathcal H^3(C',B')$, then $\uffa^*X'$ corresponds to the element $[\epsilon'\cdot(\varphi_0\times\varphi_0\times\varphi_0)]$ of $\mathcal H^3(C,B')$. 

As far as point $ii)$ of Theorem \ref{thm:XExt} is concerned, $\Gamma$ acts by composition on the left. 
One can prove that the group $\Gamma$ is canonically isomorphic to the cohomology group $\mathcal H^2(C,B',\varphi_0^*\xi')$, where the $C$-module structure  $\varphi_0^*\xi'$ on $B'$ is given by pulling back along $\varphi_0$ the action $\xi'$ of $C'$ on $B'$. The argument of the proof follows the same lines as in the lower-dimensional case of $\OPEXT$ mentioned before, see Remark \ref{rem:cat_group}. Also in this case,  all the automorphism groups of the crossed extensions lying in the same fibre turn out to be isomorphic to each other. Therefore,  $\Gamma$  is isomorphic in particular to the automorphism group of the  $0$ crossed extension:
$$
\xymatrix{
B' \ar[r]^{1}
& B' \ar[r]^{0}
& C \ar[r]^{1}
& C
}
$$
and the latter is isomorphic to $\mathcal H^2(C,B',\varphi_0^*\xi')$, as proved in \cite{CM16} in the semi-abelian case. In conclusion, we can rephrase Theorem \ref{thm:XExt} as follows:

\begin{Theorem} \label{thm:structure_Gp}
Consider two crossed extensions $X$ and $X'$, with associated elements $[\epsilon]$ in $\mathcal H^3(C,B)$ and $[\epsilon']$ in $\mathcal H^3(C',B')$ respectively, and a morphism $\uffa=(\varphi_0,\varphi)\colon (C,B)\to (C',B')$ of modules. Then
\begin{itemize}
 \item[$i)$] there exists a butterfly $\widehat{E}\colon\xymatrix@!C=3ex{X \ar[r]|-@{|} & X'}$ with $P([\widehat{E}])=\uffa$ if and only if $[\varphi\cdot \epsilon]=[\epsilon'\cdot(\varphi_0\times\varphi_0\times \varphi_0)]$;
 \item[$ii)$] if $[\BExt]_{\uffa}(X,X')\neq \emptyset$, it is a simply transitive $\mathcal H^2(C,B',\varphi_0^*\xi')$-set.
\end{itemize}
\end{Theorem}

\smallskip
Specializations of the example above give rise to well-known classification problems. For instance, let us be given two groups $K$ and $C$, and  a morphism $\psi_0\colon C \to \Out(K)$, called \emph{abstract kernel}. The obstruction problem associated with these data precisely corresponds to the classical Schreier-Mac Lane obstruction theory of non-abelian extensions, which means to find extensions of $C$ via $K$ inducing $\psi_0$. Actually, such an extension corresponds to filling the diagram below with a butterfly:
\begin{equation} \label{diag:olo}
\begin{aligned}
\xymatrix@!C=6ex{
	0 \ar[d] \ar[rr] & &  \text{Z}(K) \ar[d] \\
	0 \ar@{-->}[rd] \ar[dd] & & K \ar[dd]^{\mathcal I_K} \ar@{-->}[ld]_{k} \\
	& E \ar@{-->}[ld]_{f} \ar@{-->}[rd]|-{\textsf{conj}} \\
	C \ar[d]_1 & & \Aut(K) \ar[d] \\
	C \ar[rr]_-{\psi_0} & & \Out(K)
}
\end{aligned}
\end{equation}
where $\mathcal I_K$ is the classical crossed module of inner automorphisms of $K$ into $\Aut(K)$, its kernel being the centre $ \text{Z}(K)$ and its cokernel being the group of \emph{outer} automorphisms of $K$; the homomorphism $\mathsf{conj}$ is the restriction to $K$ of conjugation in $E$. In conclusion, applying Theorem \ref{thm:structure_Gp} to this particular case, we recover the following classical result (more details can be found in \cite{CM16}; see also \cite{CMM13} and \cite{Bourn08}).

\begin{Theorem}[Schreier-Mac Lane Classification Theorem] \label{thm:SML}
Let $\zeta_K$ denote the canonical action of $\mathrm{Out}(K)$ on $\mathrm{Z}(K)$ and $[\omega_K]$ the element of $\mathcal{H}^3(\mathrm{Out(K)},\mathrm{Z}(K),\zeta_K)$ corresponding to $\mathcal{I}_K$. Let $\psi_0\colon C \to \mathrm{Out}(K)$ be an abstract kernel and  $\mathrm{Ext}(C,K,\psi_0)$ the set of equivalence classes of extensions inducing the abstract kernel $\psi_0$. Then
\begin{itemize}
 \item[$i)$] $\mathrm{Ext}(C,K,\psi_0)\neq\emptyset$ if and only if $[\omega_K\cdot(\psi_0\times\psi_0\times \psi_0)]=0$;
 \item[$ii)$] if $\mathrm{Ext}(C,K,\psi_0)\neq\emptyset$, it is a simply transitive $\mathcal H^2(C,\mathrm{Z}(K),\psi_0^*\zeta_K)$-set.
\end{itemize}
\end{Theorem}

\subsection{Crossed extensions in the semi-abelian context} \label{sec:xext_semiab}

As already mentioned, most of the concepts involved in Section 6 carry on in any semi-abelian category satisfying \SH\ (see \cite{CM16,AMMV13,CMMV_Yoneda}). First of all, since the proof of Proposition \ref{prop:fractions} is purely formal, it can be performed in any such category, hence providing the following result, which includes also the intrinsic version of Theorem \ref{cor:Butter_fac}.

\begin{Theorem} \label{thm:structure_semiab}
Let \cC\ be a semi-abelian category satisfying \SH. The functor $Q\colon \XExt(\cC)\to [\BExt](\cC)$ provides a category of fractions of $\XExt(\cC)$ with respect to weak equivalences.
As a consequence, the fibrewise opfibration $(\Pi,\Pi_0,(\ )_0)$ admits a factorization $\Pi=P\cdot Q$ such that the resulting fibrewise opfibration $(P,P_0,(\ )_0)$ has groupoidal fibres.
$$
\xymatrix@C=6ex{
\XExt(\cC)\ar[dr]_{{\Pi}_0}\ar[r]_{Q}\ar@/^3ex/[rr]^{\Pi}
&[\BExt](\cC)\ar[d]_{P_0}\ar[r]_P
&\Mod(\cC)\ar[dl]^{(\ )_0}
\\
&\cC
}
$$
\end{Theorem}

As a consequence of Theorem \ref{thm:structure_semiab}, we obtain an intrinsic version of Theorem \ref{thm:XExt} (where \Gp\ is replaced by any semi-abelian category \cC\ with \SH). The same holds for Theorem \ref{thm:structure_Gp}, provided one considers the internal cohomology developed by Bourn and, in terms of crossed extensions, by Rodelo (see \cite{Rodelo}). Finally, note that a generalization of Theorem \ref{thm:SML} to action representable categories already appears in \cite{Bourn08}.

\section{Applications - Part III}

We end this work with some further applications of Theorem \ref{thm:structure} to categorical groups and to unital associative algebras. 

\subsection{Categorical groups} \label{sec:catgp}

It is well known that the category $\XMod(\Gp)$ of internal crossed modules in groups is equivalent to the category $\Gpd(\Gp)$ of internal groupoids in groups.
Groupoids in groups can be equivalently presented as strict categorical groups, i.e.\ strict monoidal groupoids where each object is invertible. In such terms, internal functors become strict monoidal functors. Actually strict categorical groups and strict monoidal functors organize in a 2-category, where 2-cells are monoidal natural transformations. This 2-category is equivalent to $\underline\XMod(\Gp)$, hence its bicategory of fractions $\underline{\Mon\Gp}$ is biequivalent to $\underline\Bfly(\Gp)$. In \cite{Vitale10}, it is proved that (not necessarily strict) monoidal functors serve as morphisms between strict categorical groups in $\underline{\Mon\Gp}$ and in \cite{AMMV13} it is described in detail how to perform the correspondence with butterflies. Actually, $\underline{\Mon\Gp}$ is biequivalent to the 2-category $\underline{\Cat\Gp}$ of (not necessarily strict) categorical groups, monoidal functors and monoidal natural transformations. Hence $[\Cat\Gp]$ is equivalent to $[\Bfly(\Gp)]$, hence to $[\BExt(\Gp)]$. 

Thanks to the previous equivalences, we can adapt Theorem \ref{thm:XExt} to the present context in order to formulate an obstruction and classification theorem for monoidal functors between categorical groups. With each categorical group $\mathbb G$, one can functorially associate a module $\Pi(\mathbb G)=(\Pi_0(\mathbb G),\Pi_1(\mathbb G),\xi_\bG)$, where $\Pi_0(\bG)$ and $\Pi_1(\bG)$ are the homotopy invariants of the categorical group, respectively given by the group of connected components of $\bG$ and the group of automorphisms of its identity object. Moreover, we can choose an equivalence $\Phi$ between $[\Cat\Gp]$ and $[\BExt(\Gp)]$ in such a way that $P\cdot \Phi = \Pi$, where $P \colon [\BExt(\Gp)] \to \Mod(\Gp)$ is as in Theorem \ref{cor:Butter_fac}. Via $\Phi$, one can also associate with each categorical group $\bG$ an element $[\epsilon_\bG]$ of the cohomology group $\mathcal H^3(\Pi_0(\mathbb G),\Pi_1(\mathbb G),\xi_\bG)$.

Then we are ready to recover, in the following theorem, a result on the classification of monoidal functors between categorical groups stated by Cegarra, Garc\'ia-Calcines and Ortega in \cite{CGCO} and based on the homotopical classification of categorical groups established by Sinh in \cite{Sinh}.

\begin{Theorem} \label{thm:monoidal_functors_CG}
Given two categorical groups $\mathbb H$ and $\mathbb G$, and a homomorphism of modules $\uffa=(\varphi_0,\varphi)\colon \Pi(\mathbb H)\to \Pi(\mathbb G)$, then
\begin{itemize}
\item[i)] there exists a monoidal functor $$F\colon\mathbb H \to \mathbb G$$ with $\Pi F=\uffa$ if and only if $[\varphi\cdot \epsilon_\bH]=[\epsilon_\bG\cdot(\varphi_0\times\varphi_0\times \varphi_0)]$ in the group $\mathcal H^3(\Pi_0(\mathbb H),\Pi_1(\mathbb G),\varphi_0^*\xi_\bG)$;
\item[ii)] if this is the case, isomorphism classes of such functors form a simply transitive $\mathcal H^2(\Pi_0(\mathbb H),\Pi_1(\mathbb G),\varphi_0^*\xi_\bG)$-set.
\end{itemize}
\end{Theorem}

\subsection{Crossed extensions of unital associative algebras}

In this section, we start by recalling the notion of crossed biextension of unital associative algebras over a fixed field $\bK$. Crossed biextensions were introduced in \cite{BauMi02} to give a description of the Hochschild cohomology group $\mathcal H^3_{_H}(C,B)$ for any given $C$-bimodule $B$.

A \emph{crossed bimodule} is a morphism of $A_1$-bimodules
$$
\partial\colon A_2\to A_1\,,
$$
(where $A_1$ is considered as a $A_1$-bimodule via the multiplication in $A_1$) satisfying  $$\partial(a)*a'=a*\partial(a')\,,$$ for all $a,a'\in A_2$.
Notice that the product defined by $aa'=\partial(a)*a'$ gives $A_2$ a structure of (not necessarily unital) associative algebra.

A \emph{crossed biextension} is an exact sequence
\[
X\colon\qquad
\xymatrix{
0\ar[r]
& B\ar[r]^j
& A_2\ar[r]^{\partial}
& A_1\ar[r]^p
& C\ar[r]
& 0}
\]
in $\bK$-\Vect, where $p$ is a surjective morphism of unital associative algebras and $\partial$ is a crossed bimodule. Such data determine a $C$-bimodule structure on $B$, denoted by $\Pi(X)$. Morphisms of crossed biextensions are obviously defined and they form a category $\XBiext$, with $\Pi$ becoming a functor.

In \cite{CMMV_Yoneda} it is proved that the commutative triangle
\[
\xymatrix{\XBiext\ar[rr]^\Pi\ar[dr]_{\Pi_0}&&\Bimod\ar[dl]^{(\ )_0}\\&\AssAlg_1}
\]
is a fibrewise opfibration in $\Fib(\AssAlg_1)$. In fact, the fibres of $\Pi$ are not groupoidal, so that, in order to apply Theorem \ref{thm:structure}, we need to move to the category of fractions of $\XBiext$ with respect to $\Pi$-vertical arrows. We cannot follow directly the lines of the group theoretical case (see Section \ref{sec:bfly_gp}), since $\AssAlg_1$ is not semi-abelian. However, the category $\XBiext$ is equivalent to the category $\Gpd(\AssAlg_1)$ of internal groupoids in $\AssAlg_1$ (see \cite{Ellis}). Actually, the latter is naturally endowed with a 2-category structure, having internal natural transformations as 2-cells, and the equivalence above becomes a 2-equivalence $\underline{\Gpd}(\AssAlg_1)\simeq\underline{\XBiext}$. Under this 2-equivalence, $\Pi$-vertical arrows correspond to fully faithful and essentially surjective internal functors, also called weak equivalences in \cite{MMV13}. It is proved in \cite{MMV13} that the bicategory $\underline\Frac(\cE)$ of fractions of internal groupoids in a Barr-exact category \cE\ with respect to weak equivalences has \emph{fractors} as 1-cells. The latter are a special kind of internal profunctors (see \cite{Bourn10,PTJ}) whose span representation has a fully faithful, surjective on objects, left leg. The interested reader may look at \cite{MMV13} for a more detailed account.

Thanks to the 2-equivalence between $\underline\Gpd(\AssAlg_1)$ and $\underline\XBiext$, we can describe the bicategory of fractions $\underline\BXBiext$ of crossed biextensions with respect to weak equivalences, by translating fractors into \emph{crossed bimodule butterflies} (introduced in \cite{Aldrovandi}).

\begin{Definition}
A crossed bimodule butterfly $\widehat{E}$ between two crossed biextensions $X$ and $X'$ is a diagram
\[
\xymatrix@!=2ex{
  & B \ar[d] & & B' \ar[d] \\
	& A_2 \ar[rd]^{\kappa} \ar[dd]_{\partial} & & A_2' \ar[dd]^{\partial'} \ar[ld]_{\iota} \\
  \widehat{E}\colon & & E \ar[ld]^{\delta} \ar[rd]_{\gamma} \\
	& A_1 \ar[d] & & A_1' \ar[d] \\
	& C & & C'
}
\]
such that
\begin{enumerate}
 \item[i)] $\delta$ is a surjective morphism in $\AssAlg_1$ and $\iota$ is its kernel in $k$-\Vect;
 \item[ii)] $\gamma$ is a morphism in $\AssAlg_1$, $\kappa$ a morphism in $k$-\Vect\ and $\gamma\kappa=0$;
 \item[iii)] for all $a$ in $A_2$, $a'$ in $A_2'$ and $e$ in $E$, the following conditions hold:
 \[
 \begin{array}{l}
 \iota(a'*\gamma(e))=\iota(a')e \\
 \iota(\gamma(e)*a')=e\iota(a') \\
 \kappa(a*\delta(e))=\kappa(a)e \\
 \kappa(\delta(e)*a)=e\kappa(a)
 \end{array}
 \]
\end{enumerate}
A $2$-cell $\alpha\colon \widehat{E}\Rightarrow \widehat{E}' \colon X \to X'$ is a morphism $\alpha \colon E \to E'$ in $\AssAlg_1$, commuting with the $\kappa$'s, the $\iota$'s, the $\delta$'s and the $\gamma$'s.
\end{Definition}

Notice that such an $\widehat{E}$ is, in fact, a special internal butterfly in the semi-abelian category $\AssAlg$ of (not necessarily unital) associative algebras. It is then  not surprising that also crossed bimodule butterflies admit a span representation like \eqref{diag:span_biext} as explained in \cite{Aldrovandi}. We can then repeat the argument of Proposition \ref{prop:fractions} to prove that the classifying category $[\BXBiext]$ is the category of fractions of $\XBiext$ with respect to weak equivalences.

Now, thanks to Proposition \ref{prop:factorization}, we can get the following factorization:
$$
\xymatrix@C=6ex{
\XBiext \ar[dr]_{{\Pi}_0}\ar[r]_{Q}\ar@/^3ex/[rr]^{\Pi}
&[\BXBiext]\ar[d]_{P_0}\ar[r]_P
&\Bimod\ar[dl]^{(\ )_0}
\\
&\AssAlg_1
}
$$
with $P$ a fibrewise opfibration with groupoidal fibres. Recall from \cite{BauMi02} that one can associate with each crossed biextension
\[
X\colon\qquad
\xymatrix{
0\ar[r]
& B\ar[r]^j
& A_2\ar[r]^{\partial}
& A_1\ar[r]^p
& C\ar[r]
& 0}
\]
a cocycle $\epsilon\colon C\otimes C\otimes C \to B$, hence obtaining an equivalent description of $\mathcal H_H^3(C,B)$ in terms of connected components of the fibre of $\Pi$ over $(C,B)$. Then, applying Theorem \ref{thm:structure} to the triangle $(P,P_0,(\ )_0)$, we get the next result, whose proof follows the lines of the group theoretical case.

\begin{Theorem} \label{thm:structure_AssAlg}
Consider two crossed biextensions $X$ and $X'$, with associated elements $[\epsilon]$ in $\mathcal H^3_{_H}(C,B)$ and $[\epsilon']$ in $\mathcal H^3_{_H}(C',B')$ respectively, and a morphism $\uffa=(\varphi_0,\varphi)\colon (C,B)\to (C',B')$ of bimodules. Then
\begin{itemize}
 \item[$i)$] there exists a crossed bimodule butterfly $\widehat{E}\colon\xymatrix@!C=3ex{X \ar[r]|-@{|} & X'}$ with $P([\widehat{E}])=\uffa$ if and only if $[\varphi\cdot \epsilon]=[\epsilon'\cdot(\varphi_0\otimes\varphi_0\otimes \varphi_0)]$;
 \item[$ii)$] if $[\BXBiext]_{\uffa}(X,X')\neq \emptyset$, it is a simply transitive $\mathcal H^2_{_H}(C,B')$-set, where the $C$-bimodule structure on $B'$ is defined by pulling back its $C'$-bimodule structure.
\end{itemize}
\end{Theorem}

As a final result we are going to obtain a variation of Schreier-Mac Lane Theorem \ref{thm:SML} for the case of unital associative algebras. Let us first observe that a straightforward translation to this context is not possible. Indeed, actions are not \emph{representable} in $\AssAlg_1$, i.e.\ for a given algebra $A$, in general one cannot represent actions on $A$ via morphisms into a special algebra, as it happens for $\Aut(G)$ in the case of a group $G$ (see \cite{BJK} for a detailed account on representability of actions). Nevertheless, one can rely on the intrinsic Schreier-Mac Lane theory developed in \cite{BM} and \cite{CMM13} for action accessible categories.

For each surjective $f$ in $\AssAlg_1$ and each short exact sequence
\[
\xymatrix{
0 \ar[r] & K \ar[r]^k & E \ar[r]^f & C \ar[r] & 0
}
\]
of (not necessarily unital) associative algebras, one can construct a diagram similar to \eqref{diag:olo}, where $\mathcal I_K$ is replaced by a \emph{faithful} crossed biextension and $\psi_0$ is a regular epimorphism. A crossed bimodule $\partial\colon A_2 \to A_1$ is said to be faithful if $A_1$ acts faithfully on $A_2$, i.e.\ $a$ in $A_1$ is such that $a\ast a'=a'\ast a=0$ for each $a'$ in $A_2$ if and only if $a=0$. Notice that a faithful crossed bimodule has the centre (or annihilator) $\text{Z}(A_2)$ of $A_2$ as kernel.

It is easy to see that the canonical faithful crossed extension associated with $(f,k)$ can be constructed via the quotient $q$ of $E$ over the centralizer (or annihilator) $\text{Z}(K,E)$ of $K$ in $E$. We get then the following commutative diagram
\[
\xymatrix@!C=6ex{
	0 \ar[d] \ar[rr] & &  \text{Z}(K) \ar[d] \\
	0 \ar[rd] \ar[dd] & & K \ar[dd]^{\partial} \ar[ld]_{k} \\
	& E \ar[ld]_{f} \ar[rd]^q \\
	C \ar[d]_1 & & E/\text{Z}(K,E) \ar[d] \\
	C \ar@{->>}[rr]_-{\psi_0} & & Q
}
\]
which gives rise to a crossed bimodule butterfly $\widehat{E}$ with faithful codomain, and with $\psi_0=P_0([\widehat{E}])$ a regular epimorphism. As a consequence, the action $\xi$ of $Q$ on $\mathrm{Z}(K)$ induces an action $\psi_0^*\xi$ of $C$ on $\mathrm{Z}(K)$. Notice that, given an isomorphism of crossed biextension as
\begin{equation} \label{diag:ak}
\xymatrix{
\text{Z}(K) \ar@{=}[d] \ar[r]^-j & K \ar@{=}[d] \ar[r]^-{\partial} & E/\text{Z}(K,E) \ar[d]^\sim \ar[r]^-p & Q \ar[d]^\sim_\tau \\
\text{Z}(K) \ar[r]^-j & K \ar[r]^{\partial'} & A' \ar[r]^{p'} & Q'
}
\end{equation}
the induced action $(\tau\psi_0)^*\xi'$ of $C$ on $\mathrm{Z}(K)$ coincides with $\psi_0^*\xi$.

Therefore, the extension $(f,k)$ we started with gives rise to what we call an \emph{abstract kernel}. Namely, an abstract kernel $\Psi$ is a class of isomorphism of kind \eqref{diag:ak} of diagrams of the form
\[
\xymatrix{
 & &  & C \ar@{->>}[d]^{\psi_0} \\
\text{Z}(K) \ar[r]^-j & K \ar[r]^{\partial} & A \ar[r]^{p} & Q
}
\]
with $\partial$ a faithful crossed module and $\psi_0$ regular epimorphism.

Now, fixed a representative of such an abstract kernel $\Psi$, as the one in the diagram above, a butterfly $\widehat{E}$ between the crossed biextensions $(1_C,0,0)$ and $(p,\partial,j)$ with $P([\widehat{E}])=(\psi_0,0)$ determines an extension $(f,k)$ whose associated abstract kernel is exactly $\Psi$. We can now apply Theorem \ref{thm:structure_AssAlg} in order to get the following result.

\begin{Theorem}[Schreier-Mac Lane Classification Theorem for unital associative algebras] \label{thm:SML_AA}
Given an abstract kernel $\Psi$ represented by $((p,\partial,j),\psi_0)$, let $\xi$ denote the action of $Q$ on $\mathrm{Z}(K)$ induced by the crossed module structure of $\partial$, $[\omega]$ the corresponding element of $\mathcal{H}^3_{_H}(Q,\mathrm{Z}(K),\xi)$, and $\mathrm{Ext}(C,K,\Psi)$ denote the set of isomorphism classes of extensions inducing the abstract kernel $\Psi$. Then
\begin{itemize}
 \item[$i)$] $\mathrm{Ext}(C,K,\Psi)\neq\emptyset$ if and only if $[\omega\cdot(\psi_0\otimes\psi_0\otimes \psi_0)]=0$;
 \item[$ii)$] if $\mathrm{Ext}(C,K,\Psi)\neq\emptyset$, it is a simply transitive $\mathcal H^2_{_H}(C,\mathrm{Z}(K),\psi_0^*\xi)$-set.
\end{itemize}
\end{Theorem}

\begin{Remark}
Based on the previous discussion, one can see that the set of all extensions can be obtained as the disjoint union of the sets $\mathrm{Ext}(C,K,\Psi)$ for each $\Psi$, hence providing a classification of all possible extensions of unital associative algebras with non-abelian kernel. 
\end{Remark}

\section*{Acknowledgments}

Partial financial support was received from INDAM--Istituto Nazionale di Alta Matematica ``Francesco Severi''--Gruppo Nazionale per le Strutture Algebriche, Geometriche e le loro Applicazioni.

\end{document}